\documentclass[10pt]{article}
\usepackage[utf8]{inputenc}
\usepackage{todonotes}

\usepackage{comment}
\usepackage{amsmath,amsthm,amsfonts,color}
\usepackage{amssymb}
\usepackage{graphicx}
\usepackage{setspace}
\usepackage{mathrsfs}
\usepackage{graphicx}
\usepackage[shortlabels]{enumitem}
\usepackage{authblk}
\setlist[itemize]{topsep=3pt, partopsep=6pt, itemsep=1pt, parsep=1pt}

\usepackage{calc}
\usepackage{tikz}
\usepackage{geometry}
\usepackage{pgffor}
\usepackage{ifthen}
\usetikzlibrary{arrows.meta}
\usetikzlibrary{decorations.markings}
\tikzstyle{vertex}=[circle, draw, inner sep=2pt, minimum size=6pt]
\tikzstyle{filledvertex}=[circle, draw, fill, inner sep=2pt, minimum size=6pt]

\newtheorem{theorem}{Theorem}[section]
\newtheorem{conjecture}[theorem]{Conjecture}
\newtheorem{lemma}[theorem]{Lemma}

\newtheorem{claim}{Claim}
\newtheorem{subclaim}{Claim}[claim]
\newtheorem{corollary}[theorem]{Corollary}

\let\geq\geqslant
\let\leq\leqslant

\begin{document}
\begin{spacing}{1.18}

\title{Cycle lengths in graphs of given minimum degree}
\author{Yandong Bai
\thanks{School of Mathematics and Statistics, 
Xi’an-Budapest Joint Research Center for Combinatorics,
Northwestern Polytechnical University,
Xi'an, Shaanxi 710129, China;
Research and Development Institute of 
Northwestern Polytechnical University in Shenzhen,
Shenzhen, Guangdong 518057, China.
E-mail: {\tt bai@nwpu.edu.cn}.}
\quad Andrzej Grzesik
\thanks{Faculty of Mathematics and Computer Science, Jagiellonian University, {\L}ojasiewicza 6, 30-348 Krak\'{o}w, Poland. E-mail: {\tt andrzej.grzesik@uj.edu.pl}.}
\quad Binlong Li
\thanks{School of Mathematics and Statistics, 
Xi’an-Budapest Joint Research Center for Combinatorics,
Northwestern Polytechnical University,
Xi'an, Shaanxi 710129, China.
E-mail: {\tt binlongli@nwpu.edu.cn}.}
\quad Magdalena Prorok
\thanks{AGH University of Krakow, al. Mickiewicza 30, 30-059 Krak\'ow, Poland. 
E-mail: {\tt prorok@agh.edu.pl}.}
}
\date{}

\maketitle

\begin{abstract}
    We prove that if $G$ is a 2-connected graph with minimum degree at least $k\geqslant 4$, then 
    \begin{itemize}   
       \item [(1)] $G$ contains $k$ cycles whose lengths form an arithmetic progression with common difference one or two, unless $G\cong K_{k+1}$ or $K_{k,n-k}$;
       \item [(2)] $G$ contains cycles of lengths $\ell$ modulo $k$ for all even $\ell$, unless $G\cong K_{k+1}$ or $K_{k,n-k}$;
       \item [(3)] $G$ contains cycles of lengths $\ell$ modulo $k$ for all $\ell$, unless $G\cong K_{k+1}$ or $G$ is bipartite.
    \end{itemize}    
    In addition, we show that if $k$ is even and $G$ is 2-connected with minimum degree at least $k-1$ and order at least $k+2$, then $G$ contains cycles of lengths $\ell$ modulo $k$ for all even $\ell$. As a corollary, we determine the maximum number of edges in a graph that does not contain a cycle of length divisible by $k$ for all odd $k$.
\end{abstract}

\section{Introduction}

The distribution of cycle lengths in graphs of given basic parameters, such as minimum degree, edge number, etc, has been extensively studied in the literature and remains one of the most active and fundamental research areas in graph theory. 

For a cycle (or a path) $C$, its {\em length}, denoted by $|C|$, is the number of edges in $C$. Let $k\geqslant 1$ and $\ell\geqslant 0$ be two integers. An {\em $(\ell\bmod k)$-cycle} refers to a cycle whose length is congruent to $\ell$ modulo $k$. In a graph, $k$ cycles $C_1,\ldots, C_k$ are {\em admissible} if $|C_1|, \ldots, |C_k|$ form an arithmetic progression with common difference one or two; and $k$ paths $P_1,\ldots, P_k$ are {\em admissible} if $|P_1|\geqslant 2$ and $|P_1|, \ldots, |P_k|$ form an arithmetic progression with common difference one or two. In 2022, Gao, Huo, Liu and Ma \cite{G+22} established the following remarkable results.

\begin{theorem}[Gao, Huo, Liu and Ma \cite{G+22}]\label{thm: base}
    Let $G$ be a graph with minimum degree at least $k+1$ and $k\geqslant 3$. The following statements hold: 
    \begin{itemize}
        \item [(1)] $G$ contains $k$ admissible cycles.
        \item [(2)] $G$ contains cycles of lengths $\ell$ modulo $k$ for all even $\ell$.
        \item [(3)] $G$ contains cycles of lengths $\ell$ modulo $k$ for all $\ell$ if $G$ is non-bipartite and $2$-connected. 
    \end{itemize}
\end{theorem}

\begin{theorem}[Gao, Huo, Liu and Ma \cite{G+22}]\label{thm: k-con-lmodk}
    Let $G$ be a $k$-connected graph. The following statements hold: 
    \begin{itemize}
        \item [(1)] If $k\geqslant 3$, then $G$ contains a cycle of length $0$ modulo $k$. 
        \item [(2)] If $k\geqslant 6$, then $G$ contains cycles of lengths $\ell$ modulo $k$ for all even $\ell\not\equiv 2$ modulo $k$.
    \end{itemize}
\end{theorem}

   Theorem \ref{thm: base} (1) confirms a conjecture proposed by Liu and Ma \cite{LM18} in 2018. Theorem \ref{thm: base} (2) and (3) confirm two conjectures proposed by Thomassen \cite{T83} in 1983. Theorem \ref{thm: k-con-lmodk} (1) confirms a conjecture proposed by Dean~\cite{D88} in 1988. Note that the complete graph $K_{k+1}$ and the complete bipartite graph $K_{k,n-k}$ have minimum degree $k$ but do not contain $k$ admissible cycles. Note also that, for each integer $k$, the graph $K_{k+1}$ does not contain a $(2\bmod k)$-cycle; and for each odd $k$, the graph $K_{k,n-k}$ does not contain a $(2\bmod k)$-cycle. Thus, the minimum degree conditions in the above results are tight in general. 

   In this paper, we investigate the stability of Theorems \ref{thm: base} and \ref{thm: k-con-lmodk}. Regarding Theorem \ref{thm: base} (1) and (2), we prove the following two results.

\begin{theorem}\label{thm: k adm cycles min degree k}
    Let $G$ be a $2$-connected graph with minimum degree at least $k\geqslant 4$. Then $G$ contains $k$ admissible cycles, unless $G\cong K_{k+1}$ or $K_{k,n-k}$.
\end{theorem}

\begin{theorem}\label{thm: l mod k min degree k}
    Let $G$ be a $2$-connected graph with minimum degree at least $k\geqslant 4$. Then $G$ contains cycles of lengths $\ell$ modulo $k$ for all even $\ell$, unless $G\cong K_{k+1}$ or $K_{k,n-k}$.
\end{theorem}

    Theorems \ref{thm: k adm cycles min degree k} and \ref{thm: l mod k min degree k} do not hold for $k=3$, as the Petersen graph is a counterexample. By results of Chen and Saito \cite{CS94}, and Dean et al. \cite{D+91}, the Petersen graph is the unique counterexample for $k=3$ in the context of Theorem \ref{thm: l mod k min degree k} (see Theorem \ref{thm: ChenSaitoDeanetal} in Section \ref{section: even ell}). We do not know whether there are additional counterexamples other than the Petersen graph for $k=3$ when considering Theorem \ref{thm: k adm cycles min degree k}. 

    By Theorem \ref{thm: l mod k min degree k}, we can get that the statements in Theorem \ref{thm: k-con-lmodk} still hold for 2-connected graphs with minimum degree at least $k$ for every $k\geqslant 4$.

\begin{corollary}\label{corollary: l mod k}
    Let $G$ be a $2$-connected graph with minimum degree at least $k\geqslant 4$. Then $G$ contains cycles of lengths $\ell$ modulo $k$ for all even $\ell\not\equiv 2$ modulo $k$.
\end{corollary}

    Note that $K_{k-1,n-k+1}$ does not contain a cycle of length 0 modulo $k$ if $k$ is odd. Thus, the minimum degree in Corollary \ref{corollary: l mod k} cannot be reduced for odd $k$. But for even $k$, we can improve the result as follows.

\begin{theorem}\label{thm: l mod k min degree k-1}
    Let $G$ be a $2$-connected graph with minimum degree at least $k-1$ and order at least $k+2$, where $k\geqslant 4$ is even. Then $G$ contains cycles of lengths $\ell$ modulo $k$ for all even $\ell$.
\end{theorem}

    Regarding Theorem \ref{thm: base} (3), we show the following result.

\begin{theorem}\label{thm: l mod k min degree non-bi}
    Let $G$ be a $2$-connected non-bipartite graph with minimum degree at least $k\geqslant 3$. Then $G$ contains cycles of lengths $\ell$ modulo $k$ for all $\ell$, unless $G\cong K_{k+1}$, or $k=3$ and $G$ is a Petersen graph.
\end{theorem}

   In view of graphs containing a cut-vertex such that each block is isomorphic to the complete graph $K_{k+1}$ (or $K_k$ when consider Theorem \ref{thm: l mod k min degree k-1}), one can see that 2-connectivity cannot be dropped in the above theorems. 

    By Theorem \ref{thm: l mod k min degree k}, we can also determine the maximum number of edges in an $n$-graph that does not contain a cycle of length 0 modulo $k$ for odd $k$. 
    
    Let $\mathcal{F}$ be a set of graphs. A graph is \textit{$\mathcal{F}$-free} if it contains no subgraph isomorphic to any member of $\mathcal{F}$. The classic Tur\'{a}n problem in graphs seeks to determine the maximum number of edges in $\mathcal{F}$-free graphs on $n$ vertices. This maximum number is referred to as the \textit{Tur\'{a}n number} for $\mathcal{F}$ and is commonly denoted by $ex(n,\mathcal{F})$, i.e.,
    \[
       ex(n,\mathcal{F})
       =\max \{|E(G)|:\ |V(G)|=n, ~G~\text{is}~\mathcal{F}\text{-free}\}.
    \]
    For two integers $k>\ell\geqslant 0$, denote the set of all $(\ell \bmod k)$-cycles by $\mathcal{C}_{\ell\bmod k}$. The problem of determining the Tur\'{a}n number for $(\ell \bmod k)$-cycles has a research history spanning nearly half a century. In 1976, Burr and Erd\H{o}s~\cite{Erd76} proposed this problem and conjectured that $ex(n,\mathcal{C}_{\ell \bmod k})=\varTheta(n)$ for $k>\ell\geqslant 0$ when $k\mathbb{Z}+\ell$ contains an even integer. In view of the complete balanced bipartite graphs, one can see that if $k\mathbb{Z}+\ell$ contains no even integers, then $ex(n,\mathcal{C}_{\ell \bmod k})\geqslant\lfloor n^2/4\rfloor$. Thus, the condition that $k\mathbb{Z}+\ell$ contains an even integer cannot be dropped. Bollob\'{a}s \cite{Bol77} confirmed the above conjecture and demonstrated that $ex(n,\mathcal{C}_{\ell \bmod k})\leqslant((k+1)^k-1)n/4k$. Subsequently, Erd\H{o}s proposed the problem of determining the smallest constant $c_{\ell,k}$ such that every $n$-vertex graph with at least $c_{\ell,k}\cdot n$ edges contains an $(\ell \bmod k)$-cycle. For $k>\ell\geqslant 3$, Sudakov and Verstra\"{e}te \cite{SV17} showed that 
    \[
       \frac{ex(k,C_{\ell})}{k}
       \leqslant c_{\ell,k}
       \leqslant 96\cdot \frac{ex(k,C_{\ell})}{k}.
    \]
    It follows that for even $\ell\geqslant 4$, determining $c_{\ell,k}$ is at least as hard as the famous extremal problem of determining the Tur\'{a}n number for the even cycle $C_{\ell}$. 

    To the best of our knowledge, precise values of $c_{\ell,k}$ are known for very few $k,\ell$. If $G$ contains no cycles, or alternatively, no (0 mod 1)-cycles, then one can see that $e(G)\leqslant n-1$ and the equality holds if and only if $G$ is a tree. It is not difficult to see that if $G$ contains no (0 mod 2)-cycles then $e(G)\leqslant 3(n-1)/2$ and, for $2|(n-1)$, the equality holds if and only if each block of $G$ is a triangle. The above two results imply that $c_{0,1}=1$ and $c_{0,2}=3/2$. The first non-trivial result on $c_{\ell,k}$ was obtained by Chen and Saito \cite{CS94}, who showed that $c_{0,3}=2$. It has been shown in \cite{CS94} that if $G$ contains no (0 mod 3)-cycles then $e(G)\leqslant 2(n-2)$ for $n \geqslant 3$ and the equality holds if and only if $G$ is isomorphic to $K_{2,n-2}$. 

\begin{theorem}[Chen and Saito \cite{CS94}]\label{thm: 0 mod 3}
   If $G$ is an $n$-vertex graph such that every vertex (with at most one exception) has degree at least $3$, then $G$ contains a $(0\bmod 3)$-cycle. As a corollary, for $n \geqslant 3$,
   \[
   ex(n,\mathcal{C}_{0\bmod 3})=2(n-2).
   \]
\end{theorem}

    Dean, Kaneko, Ota and Toft \cite{D+91} and, independently, Saito \cite{Sai92} showed that if $G$ contains no (2 mod 3)-cycles then $e(G)\leqslant 3(n-3)$ for $n\geqslant 5$, and the equality holds if and only if $G$ is isomorphic to $K_{3,n-3}$. Gao, Li, Ma and Xie \cite{GLMX} showed that if $G$ contains no (2 mod 4)-cycles then $e(G)\leqslant 5(n-1)/2$ and, for $4|(n-1)$, the equality holds if and only if each block of $G$ is isomorphic to $K_5$. Bai, Li, Pan and Zhang \cite{B+25} showed that if $G$ contains no (1 mod 3)-cycles then $e(G)\leqslant 5(n-1)/3$ and, for $9|(n-1)$, the equality holds if and only if each block of $G$ is isomorphic to the Petersen graph.

    For the case of $(0\bmod 4)$-cycles, Gy\H{o}ri et al. \cite{G+23} obtained the following result.  

\begin{theorem}[Gy\H{o}ri et al. \cite{G+23}]\label{thm: 0 mod 4}
    If $G$ is an $n$-vertex graph without $(0 \bmod 4)$-cycles, then $e(G)\leqslant \lfloor19(n-1)/12\rfloor$ and the bound is tight for all $n\geqslant 1$, i.e., 
    \[
     ex(n,\mathcal{C}_{0\bmod 4})=\left\lfloor\frac{19}{12}(n-1)\right\rfloor.
    \]  
\end{theorem}

It follows that all precise values of $c_{\ell,k}$ for $k\leqslant 4$ are fixed, as also shown in Table \ref{tab:clk}. However, prior to this work, not too much was known for the value of $c_{\ell,k}$ for large $k$. 

\begin{table}[ht]
    \renewcommand{\arraystretch}{1.5}
	\centering
	\begin{tabular}{l|l|l}
		\hline \hline
		$c_{0,2}=3/2$ (trivial)                   & ---                        & ---                                              \\ \hline
		$c_{0,3}=2$ (\cite{CS94})             & $c_{1,3}=5/3$ (\cite{B+25}) & $c_{2,3}=3$ (\cite{D+91,Sai92}) \\ \hline
		$c_{0,4}=19/12$ (\cite{G+23}) & ---                                      & $c_{2,4}=5/2$ (\cite{GLMX})       \\ \hline \hline
	\end{tabular}
	\caption{Precise values of $c_{\ell,k}$ for $k\leqslant 4$.}
	\label{tab:clk}
\end{table}

In this paper, we bound the maximum number of edges in a graph that does not contain a cycle of length $\ell$ modulo~$k$ for all integers $\ell$ and $k\geqslant 3$. 

\begin{theorem}\label{thm: general bound}
    Let $k\geq 3$ and $\ell\geq 0$ be integers such that $k> \ell$ and $k\mathbb{Z}+\ell$ contains an even integer. Then
    \begin{itemize}
        \item [(1)] $ex(n,\mathcal{C}_{\ell\bmod k})\leqslant (k-1)(n-k+1)$ if $n \geq 2k-1$ and $\ell\neq 2$,
        \item [(2)] $ex(n,\mathcal{C}_{2\bmod k})\leqslant k(n-k)$ if $n \geq 2k+1$.
    \end{itemize}
\end{theorem}

Note that for odd $k$, the complete bipartite graph $K_{{k-1},{n-k+1}}$ contains $(k-1)(n-k+1)$ edges and no ($0\bmod k$)-cycles, while $K_{{k},{n-k}}$ contains $k(n-k)$ edges and no ($2\bmod k$)-cycles, so the bound on the number of edges in Theorem~\ref{thm: general bound} is tight for odd $k$ and $\ell = 0$ or $2$. As a consequence, we deduce that $c_{0,k}=k-1$ and $c_{2,k}=k$ for odd $k$. 
Note also that for $n<2k-1$ the existence of a $(0\bmod k)$-cycle is equivalent to the existence of a $k$-cycle, while for $n<2k+1$ the existence of a $(2\bmod k)$-cycle is equivalent to the existence of a $(k+2)$-cycle. Since $ex(n,C_k)$ for odd $k$ is determined by F\"uredi and Gunderson \cite{FG15}, Theorem~\ref{thm: general bound} implies the following result. 

\begin{corollary}\label{thm:c0k c2k}
    Let $k\geqslant 3$ be an odd integer. Then
    \[
    ex(n,\mathcal{C}_{0 \bmod k})= 
    \begin{cases}
      \binom{n}{2}, & n \leqslant k-1; \\
      \binom{k-1}{2}+\binom{n-k+2}{2}, &  k \leqslant n \leqslant 2k-4; \\
      (k-1)(n-k+1), & n\geqslant 2k-3.
    \end{cases}
    \]
and
    \[
    ex(n,\mathcal{C}_{2 \bmod k})= 
    \begin{cases}
      \binom{n}{2}, & n \leqslant k+1; \\
      \binom{k+1}{2}+\binom{n-k}{2}, &  k+2 \leqslant n \leqslant 2k-1; \\
      k(n-k), & n\geqslant 2k.
    \end{cases}
    \] 
\end{corollary}

For $\ell=0$ and even $k$ we know $ex(n,\mathcal{C}_{0\bmod k})$ only when $k \leq 4$. For the remaining values of $k$ we state the following conjecture.

\begin{conjecture}\label{conj: 0 mod even k}
    Let $k\geqslant 6$ be an even integer and $n\geqslant k-1$ an integer. Then
    \[
        ex(n,\mathcal{C}_{0 \bmod k})\leqslant \frac{k-1}{2}(n-1),
    \]
    and graphs in which each block is isomorphic to $K_{k-1}$ form a unique class of extremal graphs when $n-1$ is divisible by $k-2$. 
\end{conjecture}

If $\ell=1$, then $k$ must be odd and we suppose the following.

\begin{conjecture}\label{conj: 1 mod odd k}
    Let $k\geqslant 5$ be an odd integer and $n\geqslant k$ an integer. Then
    \[
        ex(n,\mathcal{C}_{1 \bmod k})\leqslant \frac{k}{2}(n-1),
    \]
    and graphs in which each block is isomorphic to $K_{k}$ form a unique class of extremal graphs when $n-1$ is divisible by $k-1$.  
\end{conjecture}

For $\ell=2$ and even $k$ we know $ex(n,\mathcal{C}_{2\bmod k})$ only when $k \leq 4$. For the remaining values of $k$, Gao, Li, Ma and Xie~\cite{GLMX} conjectured that $c_{2,k}=(k+1)/2$, so we propose the following.

\begin{conjecture}\label{conj: 2 mod even k}
    Let $k\geqslant 6$ be an even integer and $n\geqslant k+1$ an integer. Then
    \[
        ex(n,\mathcal{C}_{2 \bmod k})\leqslant \frac{k+1}{2}(n-1),
    \]
    and graphs in which each block is isomorphic to $K_{k+1}$ form a unique class of extremal graphs when $n-1$ is divisible by $k$. 
\end{conjecture}

As we mentioned earlier, determining $c_{\ell,k}$ for general even $\ell$ is as hard as determining Tur\'an number for $C_\ell$. Therefore, we conjecture its value only for odd $\ell$, and consequently odd $k$. 

\begin{conjecture}
    Let $k > \ell\geqslant 3$ be two odd integers and $n \geq k+\ell$ an integer. Then
    \[
        ex(n,\mathcal{C}_{\ell\bmod k})\leqslant \frac{k+\ell-2}{2}\left(n-\frac{k+\ell-2}{2}\right),
    \]
    and a complete bipartite graph with $(k+\ell-2)/2$ vertices on one side is the unique extremal graph.
\end{conjecture}

   The rest of this paper is organized as follows. In Section \ref{section: pre}, we give some necessary notations and definitions as well as some useful results needed in our proofs. In Section~\ref{section: adm cycles} we prove Theorem~\ref{thm: k adm cycles min degree k}, in Section~\ref{section: even ell} Theorems~\ref{thm: l mod k min degree k} and \ref{thm: l mod k min degree k-1}, in Section~\ref{section: all ell} Theorem~\ref{thm: l mod k min degree non-bi}, while Section~\ref{section: edge bound} is devoted to the proof of Theorem~\ref{thm: general bound}.

\section{Notations and preliminaries}\label{section: pre}

    All graphs considered in this paper are finite, undirected and simple. Let $G$ be a graph and $v\in V(G)$. The {\em degree} of $v$, denoted by $d(v)$, is the number of its neighbors in $G$, whose set is denoted by $N(v)$. For a vertex set $X\subseteq V(G)$, we denote $N_X(v) = N(v) \cap X$ and $d_X(v) = |N_X(v)|$. Additionally, for a set $U \subseteq V(G)$, by $N_X(U)$ we denote the union of $N_X(v)$ over all $v \in U$. The minimum degree of a graph $G$, denoted by $\delta(G)$, is the minimum number among all the degrees of the vertices in~$G$. 
    
    A {\em bipartite} graph is a graph whose vertex set can be partitioned into two parts such that the two end-vertices of each edge of the graph are in different parts. For a vertex set $X\subseteq V(G)$, denote by $D[X]$ the subgraph of $G$ induced by $X$. For two disjoint vertex sets $X,Y\subseteq V(G)$, denote by $G(X,Y)$ the subgraph induced by all the edges between $X$ and $Y$ in $G$. 
    A {\em $\varTheta$-graph} is graph consisting of 3 paths having common end-vertices and disjoint internal vertices.
    
    For a cycle $C$ in a graph $G$ and two vertices $x, y \in V(C)$, by $d_C(x,y)$ we denote the distance between $x$ and $y$ in~$C$. If we fix an orientation of $C$, then by $\overrightarrow{C}[x,y]$ we denote the path from $x$ to $y$ along the orientation of $C$, and by $\overleftarrow{C}[x,y]$ the path from $y$ to $x$ along the orientation of $C$.
    For a path $P$ in $G$ and two vertices $x, y \in V(P)$, by $P[x,y]$ we denote the path from $x$ to $y$ contained in $P$. 
    
    A graph is {\em connected} if there exists a path between any two vertices in the graph. For a positive integer $k$, call a graph {\em $k$-connected} if it has order at least $k+1$ and the deletion of any set of at most $k-1$ vertices results in a connected subgraph. The {\em connectivity} of a graph $G$ is the maximum $k$ such that $G$ is $k$-connected. A {\em component} of a graph $G$ is a maximal connected subgraph of $G$. A subset $S$ of $V(G)$ is a {\em vertex cut} of $G$ if $G-S$ has more components than $G$. If a vertex cut consists of exactly one vertex $x$, then $x$ is a {\em cut-vertex} of $G$. A {\em block} of $G$ is a maximum connected subgraph of $G$ with no cut-vertex. In other words, a block is an isolated vertex, an edge or a maximal 2-connected subgraph. An {\em end block} of $G$ is a block of $G$ containing exactly one cut-vertex of $G$. An {\em inner-vertex} in an end block is a vertex distinct with the cut-vertex contained in the end block. 
    A graph is a {\em block-chain} if it has exactly two end blocks.
    We view a 2-connected graph as its unique end block. In particular, every vertex in a 2-connected graph is an inner-vertex. 

    For two vertices $x,y\in V(G)$, an \textit{$(x,y)$-path} is a path with two end-vertices $x,y$. For two subset $X,Y\subset V(G)$, an \textit{$(X,Y)$-path} is an $(x,y)$-path with $x\in X$, $y\in Y$, and all internal vertices in $V(G)\backslash(X\cup Y)$. An \textit{$X$-path} is a path with two end-vertices in $X$ and all internal vertices in $V(G)\backslash X$.
    For an integer $\ell$, we use $\ell$-path (cycle) to denote a path (cycle) with length $\ell$. The following results on admissible paths will be useful in our proofs.

\begin{theorem}[Gao, Huo, Liu and Ma \cite{G+22}]\label{thm: k adm paths}
    Let $k\geqslant 1$ be an integer and $G$ a graph with two distinct vertices $x,y$ such that $G+xy$ is $2$-connected. If $d(v)\geqslant k+1$ for each $v\in V(G)\backslash \{x,y\}$, then $G$ contains $k$ admissible $(x,y)$-paths.
\end{theorem}

    The following result shows that we can have one more vertex with degree less than $k+1$ for the existence of $k$ admissible paths in a 2-connected graph, which relaxes the degree condition in Theorem \ref{thm: k adm paths}. 

\begin{theorem}[Chiba, Ota and Yamashita \cite{COY22}]\label{thm: k adm paths 2}
    Let $k\geqslant 1$ be an integer and $G$ a graph on at least four vertices with $x,y,z\in V(G)$ such that $x\neq y$ and $G+xy$ is $2$-connected. If $d(v)\geqslant k+1$ for each $v\in V(G)\backslash \{x,y,z\}$, then $G$ contains $k$ admissible $(x,y)$-paths.
\end{theorem}

    We first show the following two useful lemmas, which roughly state that, if a 2-connected graph $G$ with bounded minimum degree does not contain $k$ admissible cycles or a cycle of length $\ell$ modulo $k$ for some even $\ell$, then the connectivity of $G$ increases with $k$.

\begin{lemma}\label{lemma: high connectivity no k adm cycles}
    Let $G$ be a $2$-connected graph with minimum degree at least $k-r$, where $k\geqslant 2$ and $r\geqslant 0$ are integers. If $G$ does not contain $k$ admissible cycles, then the connectivity of $G$ is at least $(k+2)/2-r$.
\end{lemma}

\begin{proof}
    Let $V_0$ be a minimum vertex cut of $G$ with $n_0=|V_0|\geqslant 2$. Assume the opposite that $n_0\leqslant (k+1)/2-r$. Let $Q_1,Q_2$ be two components of $G-V_0$ and $x,y$ two vertices in $V_0$. We claim that $G[V(Q_1)\cup \{x,y\}]+xy$ is 2-connected. If not, say $w$ is a cut-vertex in it, then $(V_0\backslash \{x,y\})\cup \{w\}$ is a smaller vertex cut of $G$, a contradiction. Note that every vertex other than $x,y$ in $G[V(Q_1)\cup \{x,y\}]$ has degree at least $k-r-n_0+2$. Note also that $k-r-n_0+2\geqslant 2$. By Theorem \ref{thm: k adm paths}, there exist $k-r-n_0+1$ admissible $(x,y)$-paths in $G[V(Q_1)\cup \{x,y\}]$. Similarly, we can get $k-r-n_0+1$ admissible $(x,y)$-paths in $G[V(Q_2)\cup \{x,y\}]$. Concatenating these admissible paths, we get $2(k-r-n_0+1)-1\geqslant k$ admissible cycles in $G$, a contradiction.
\end{proof}

\begin{lemma}\label{lemma: high connectivity no l mod k cycle}
    Let $G$ be a $2$-connected graph with minimum degree at least $k-r$, where $k\geqslant 2$ and $r\geqslant 0$ are integers. If $G$ does not contain cycles of length $\ell$ modulo $k$ for some even $\ell$, then the connectivity of $G$ is at least $(k+2)/2-r$. 
\end{lemma}

\begin{proof}
    Let $V_0$ be a minimum vertex cut of $G$ with $n_0=|V_0|\geqslant 2$. Assume the opposite that $n_0\leqslant (k+1)/2-r$. Since $G$ does not contain an $(\ell\bmod k)$-cycle for some even $\ell$, we get that $G$ does not contain $k$ cycles of consecutive lengths. By Lemma \ref{lemma: high connectivity no k adm cycles}, $G$ contains $k$ admissible cycles of common difference two. If $G$ is bipartite, then we are done as now $G$ contains $k$ cycles of consecutive even lengths. So $G$ is non-bipartite. If $k$ is odd, then $G$ contains an $(\ell\bmod k)$-cycle for any even $\ell$, a contradiction. So $k$ is even and $n_0\leqslant k/2-r$. If $r\geqslant (k-2)/2$, then $(k+2)/2-r\leqslant 2$ and the assertion clearly holds. So assume that $r\leqslant (k-4)/2$. 
    
    Let $Q_1,Q_2$ be two components of $G-V_0$. Similar to the analysis in the proof of Lemma \ref{lemma: high connectivity no k adm cycles}, we have that for any two distinct vertices $x,y\in V_0$, there are $k-r-n_0+1$ admissible $(x,y)$-paths in $G[V(Q_i)\cup \{x,y\}]$, $i\in \{1,2\}$. If the $k-r-n_0+1$ admissible $(x,y)$-paths in $G[V(Q_i)\cup \{x,y\}]$ are of consecutive lengths, then $G$ have $k$ cycles of consecutive lengths, a contradiction. So we conclude that for each two vertices $x,y\in V_0$ and $i\in\{1,2\}$, $G[V(Q_i)\cup \{x,y\}]$ has $k-r-n_0+1\geqslant (k+2)/2$ admissible $(x,y)$-paths with common difference two. 
    
    If $G[V(Q_1)\cup V_0]$ contains an odd cycle $C$, then by $G$ being 2-connected there exist two disjoint paths from $C$ to $\{x,y\}\subseteq V_0$, implying that there are two $(x,y)$-paths in $G[V(Q_1)\cup V_0]$ with different parities. Note that $G[V(Q_2)\cup \{x,y\}]$ contains $(k+2)/2$ admissible $(x,y)$-paths of lengths with common difference two. It follows that $G$ contains $(k+2)/2$ cycles of consecutive even lengths, a contradiction. So we conclude that $G[V(Q_1)\cup V_0]$ is bipartite, and similarly, $G[V(Q_2)\cup V_0]$ is bipartite. 
    
    If $n_0\geqslant 3$, then there exist $x,y\in V_0$ such that $x,y$ are in the same part of $G[V(Q_1)\cup V_0]$ if and only if $x,y$ are in the same part of $G[V(Q_2)\cup V_0]$. This implies the existence of $k$ admissible even cycles, a contradiction. So we conclude that $n_0=2$. 
    
    Let $V_0=\{x,y\}$. Since $\delta(G)\geqslant k-r\geqslant (k+4)/2\geqslant 3$, we can assume w.l.o.g. that $x$ has at least two neighbors in $G[V(Q_1)\cup\{x,y\}]$. Note that $G[V(Q_1)\cup\{x,y\}]+xy$ is 2-connected. If $G[V(Q_1)\cup\{x,y\}]$ is not 2-connected, then it is a block-chain such that $x$ and $y$ are inner-vertices of two distinct end blocks. In this case we let $G_1$ be the end block of $G[V(Q_1)\cup\{x,y\}]$ containing $x$ (so $|V(G_1)|\geqslant 4$), $y'$ be the cut-vertex of $G[V(Q_1)\cup\{x,y\}]$ contained in $G_1$, and $z$ be a neighbor of $x$ in $G_1$. If $G[V(Q_1)\cup\{x,y\}]$ is 2-connected, then let $G_1=G[V(Q_1)\cup\{x,y\}]$, $y'=y$ and $z$ be a neighbor of $x$ in $G_1$. Now $G_1$ is bipartite and every vertex in $V(G_1)\backslash\{x,y',z\}$ has degree at least $k-r$ in $G_1$. By Theorem \ref{thm: k adm paths 2}, $G_1$ has $k-r-1$ admissible $(x,z)$-paths. Together with the edge $xz$, we can get $k-r-1\geqslant k/2$ admissible cycles, which are consecutive even cycles, a contradiction.
\end{proof}

The following simple lemma will be used many times in our proof.

\begin{lemma}\label{lemma: concatenating paths}
    Let $G$ be a graph, $H$ a subgraph of $G$ and $X,Y$ two disjoint subsets of $V(H)$. Assume that $A=\{a_1,\ldots,a_s\}$, $s\geq 1$, is a set of positive integers forming an arithmetic progression, such that for every $x\in X$, $y\in Y$, $i\in [1,s]$, $H$ contains an $(x,y)$-path of length $a_i$. Assume also that $B=\{b_1,\ldots,b_t\}$, $t\geq 1$, is a set of positive integers forming an arithmetic progression with $b_1\geqslant 2$, such that for every $j\in [1,t]$, $G$ contains an $(X,Y)$-path of length $b_j$ with all internal vertices in $V(G)\backslash V(H)$. 
    \begin{itemize}
        \item[(1)] If each of $A$ and $B$ has common difference $1$ or $2$, then $G$ has $s+t-1$ admissible cycles.
        \item[(2)] If the common differences of $A$ and $B$ are $2$ and $1$, respectively, then $G$ has $2s+t-2$ admissible cycles.
        \item[(3)] If the common differences of $A$ and $B$ are $1$ and $2$, respectively, then $G$ has $s+2t-2$ admissible cycles.
    \end{itemize}
\end{lemma}

\begin{proof}
    Consider an arbitrary $(X,Y)$-path $P$ with all internal vertices in $V(G)\backslash V(H)$, say from $x\in X$ to $y\in Y$ and denote by $p=|P|\geqslant 2$. For any $i\in [1,s]$, since $H$ contains an $(x,y)$-path of length $a_i$, by concatenating such a path with $P$, we obtain a cycle of length $p+a_i$, implying $s$ cycles of lengths $p+a_1$, $p+a_2$, \ldots, $p+a_s$. Recall that, for any $j\in [1,t]$, $G$ contains an $(X,Y)$-path of length $b_j$ with all internal vertices in $V(G)\backslash V(H)$. So we can obtain cycles of all lengths in $\{a_i+b_j: i\in [1,s], j\in [1,t]\}:=L$. One can see that $L$ is an arithmetic progression with common difference 1 or 2, and
    \[
      |L| = 
      \begin{cases}
      s+t-1,    & \text{if the common differences of $A,B$ are both $1$ or both $2$}, \\
      2s + t-2, & \text{if the common differences of $A,B$ are $2$ and $1$, respectively}, \\
      s+2t-2,   & \text{if the common differences of $A,B$ are $1$ and $2$, respectively}.
      \end{cases}
    \]
    The statements then follow from the above observation and the fact that $2s+t-2\geqslant s+t-1$ for $s\geqslant 1$ and $s+2t-2\geqslant s+t-1$ for $t\geqslant 1$.
\end{proof}

\begin{lemma}\label{lemma: concatenating paths X}
    Let $G$ be a graph, $H$ a subgraph of $G$, and $X\subseteq V(H)$. Assume that $A=\{a_1,\ldots,a_s\}$, $s\geqslant 1$, is a set of positive integers forming an arithmetic progression, such that for every two distinct vertices $x,y\in X$ and $i\in[1,s]$, $H$ contains an $(x,y)$-path of length $a_i$. Assume also that $B=\{b_1,\ldots,b_t\}$, $t\geqslant 1$, is a set of positive integers forming an arithmetic progression with $b_1\geqslant 2$, such that for every $j\in [1,t]$, $G$ contains an $X$-path of length $b_j$ with all internal vertices in $V(G)\backslash V(H)$. 
    \begin{itemize}
        \item[(1)] If each of $A$, $B$ has common difference $1$ or $2$, then $G$ has $s+t-1$ admissible cycles.
        \item[(2)] If the common differences of $A,B$ are $2$ and $1$, respectively, then $G$ has $2s+t-2$ admissible cycles.
        \item[(3)] If the common differences of $A,B$ are $1$ and $2$, respectively, then $G$ has $s+2t-2$ admissible cycles.
    \end{itemize}
\end{lemma}

\begin{proof}
    The proof is similar to that of Lemma \ref{lemma: concatenating paths}, and we omit the details.
\end{proof}

\section{Admissible cycles}\label{section: adm cycles}

   In this section, we set 
   \[ 
      \mathcal{F}:=\{K_{s,t}: \min\{s,t\}\geqslant2\}\cup \{K_{s,t}^-: \min\{s,t\}\geqslant2\text{ and }\max\{s,t\}\geqslant3\}, 
   \]
   where $K_{s,t}^-$ is obtained from $K_{s,t}$ by removing an arbitrary edge. 

\begin{lemma}\label{lemma: s-con-odd-paths}
    Let $F$ be a graph in $\mathcal{F}$ with bipartition $(S,T)$. Set $s=|S|$, $t=|T|$ and $r=\min\{s,t\}$. For any two distinct vertices $u,v\in V(F)$, the following statements hold: 
    \begin{itemize}
        \item[(1)] If $u,v\in S$ (or $u,v\in T$), then $F$ has $(u,v)$-paths of consecutive even lengths from $2$ to $2r-2$. Moreover, if $s>r$ (or $t>r$), then $F$ has a $(u,v)$-path of length $2r$ as well.
        \item[(2)] If $u\in S$ and $v\in T$, then $F$ has $(u,v)$-paths of consecutive odd lengths from $1$ to $2r-1$, except that $uv\notin E(G)$, in which case $F$ has $(u,v)$-paths of consecutive odd lengths from $3$ to $2r-1$.
    \end{itemize}
\end{lemma}

\begin{proof}
    It is not difficult to check the assertion for $F\cong K_{s,t}$. So assume that $F\cong K_{s,t}^-$, say $xy\notin E(F)$ with $x\in S$ and $y\in T$. One can check that the assertion holds for $r\leqslant 3$. So assume that $r\geqslant 4$. Let $F'=F+xy$ and $P=w_0w_1\ldots w_{p}$ an arbitrary $(u,v)$-path in $F'$, where $u=w_0$, $v=w_{p}$ and $p\geqslant 2$. It suffices to show the existence of a $(u,v)$-path of length $p$ in $F$. If $P$ does not pass through $xy$, then $P$ itself is a required $(u,v)$-path. So assume that $P$ passes through $xy$. Since $p\geqslant 2$, we can assume w.l.o.g. that $v\notin\{x,y\}$ and that $y$ is the successor of $x$ in $P$. Let $x'$ be the successor of $y$. If $S\backslash \{x,x',u,v\}\neq\emptyset$, then let $x^*\in S\backslash \{x,x',u,v\}$. Replacing $x$ with $x^*$ (if $x^*\notin V(P)$) or exchanging $x$ and $x^*$ (if $x^*\in V(P)$) yields a required $(u,v)$-path. So assume that $S\backslash \{x,x',u,v\}=\emptyset$. It follows that $s=r=4$ and $u,v\in S$. Thus, $T\backslash V(P)\neq \emptyset$. Let $y^*\in T\backslash V(P)$. Then replacing $y$ with $y^*$ in $P$ yields a required $(u,v)$-path.
\end{proof}

We recall the following result from \cite{GHM21} and prove a useful lemma.

\begin{theorem}[Gao, Huo and Ma \cite{GHM21}]\label{thm: 2-con-K3-k-cycles}
    Let $G$ be a $2$-connected graph with minimum degree at least $k\geqslant 2$. If $G$ contains a triangle, then $G$ contains $k$ cycles of consecutive lengths, unless $G\cong K_{k+1}$.
\end{theorem}

\begin{lemma}\label{LeGxy2-connected}
    Let $G$ be a graph with $x,y\in V(G)$ such that $G-\{x,y\}$ is connected and every end block of $G-\{x,y\}$ has an inner-vertex adjacent to either $x$ or $y$. If there exist two disjoint edges from $\{x,y\}$ to $V(G)\backslash\{x,y\}$, then $G+xy$ is $2$-connected.
\end{lemma}

\begin{proof}
    Clearly, $G+xy$ is connected. Assume that $G+xy$ is not $2$-connected. Then $G+xy$ has a cut-vertex $z$ in $\{x,y\}$ or in $V(G)\backslash \{x,y\}$. Since $G-\{x,y\}$ is connected and there exist two disjoint edges from $\{x,y\}$ to $V(G)\backslash\{x,y\}$, we get that $z$ is in $V(G)\backslash\{x,y\}$ and it is also a cut-vertex of $G-\{x,y\}$, a contradiction to the fact that any end block of $G-\{x,y\}$ has an inner-vertex adjacent to $x$ or $y$.
\end{proof}

   The following lemma states that containing a subgraph $F \in \mathcal{F}$, for example a 4-cycle, is beneficial for finding admissible cycles. 

\begin{lemma}\label{lemma: 2-con-bi-H}
    Let $G$ be a $2$-connected graph with minimum degree at least $k\geqslant 3$ that contains $F$ which is isomorphic to a graph in $\mathcal{F}$. Let $(S,T)$ be the bipartition of $V(F)$ and set $s=|S|$, $t=|T|$, $r=\min\{s,t\}$. If each vertex in $V(G)\backslash V(F)$ has at most $r-1$ neighbors in $F$, then $G$ contains $k$ admissible cycles, unless $G\cong K_{k+1}$ or $K_{k,n-k}$.
\end{lemma}

\begin{proof}
   We prove by contradiction. By Lemma \ref{lemma: high connectivity no k adm cycles} and $k\geqslant 3$, we get that $G$ is 3-connected. If $G$ contains a triangle, then by Theorem \ref{thm: 2-con-K3-k-cycles}, $G$ contains $k$ cycles of consecutive lengths or $G\cong K_{k+1}$, a contradiction. So $G$ contains no triangle. Let $F\in \mathcal{F}$ be a subgraph of $G$ with minimum order satisfying the assumptions of the lemma. Recall that $F\cong K_{s,t}$ or $K_{s,t}^-$, implying that $F$ is an induced subgraph of $G$. 
   
   If $V(G)=V(F)$, then by the minimum degree condition, we have that $r\geqslant k$ and $G=F$. If $r\geqslant k+1$, then any edge $xy\in E(F)$ together with $(x,y)$-paths of consecutive odd lengths from 3 to $2r-1$ (by Lemma \ref{lemma: s-con-odd-paths}), form $r-1\geq k$ admissible cycles in $G$, a contradiction. So $r=k$ and $G\cong K_{k,n-k}$ (by the minimum degree condition), a contradiction. So we assume that $V(G)\backslash V(F)\neq\emptyset$. Let $Q$ be a component of $G-V(F)$. By the 2-connectivity of $G$, there are two disjoint edges from $Q$ to $F$. 
    
\begin{claim}\label{claim: at least 2 neighbors in H}
    $r\geqslant 3$.
\end{claim}

\begin{proof}
    Assume the opposite, say, $s=2$. Then every vertex in $V(G)\backslash V(F)$ has at most one neighbor in $F$. If $F\cong K_{2,t}^-$ for some $t\geqslant 3$, say $\{x_1,x_2\}=S$, $y_1\in T$ and $x_1y_1\notin E(F)$, then $x_2y_1\in E(F)$. It follows that $y_1$ has exactly one neighbor in $F$ and we can replace $F$ with $F-y_1$ (which is isomorphic to $K_{2,t-1}$), a contradiction to the minimality of $V(F)$. So $F\cong K_{2,t}$ for some $t\geqslant 2$. Recall that $G$ is 3-connected, implying that $N_T(Q)\neq\emptyset$.

    First assume that $N_S(Q)\neq \emptyset$. Let $G_1$ be the graph obtained from $Q$ by adding two vertices $x^*$ and $y^*$, and for each vertex $v\in V(Q)$, 
    \begin{itemize}
	    \item[(1)] adding an edge $vx^*$ if $v$ has a neighbor in $S$, and
        \item[(2)] adding an edge $vy^*$ if $v$ has a neighbor in $T$.
    \end{itemize}
    Since $G$ is triangle-free and 2-connected, there are two disjoint edges from $Q$ to $S$ and $T$, respectively, and every end block of $Q$ has an inner-vertex adjacent to $S$ or $T$. It follows that $G_1$ has two disjoint edges from $Q$ to $\{x^*,y^*\}$ and every end block of $Q$ has an inner-vertex adjacent to $x^*$ or $y^*$. By Lemma \ref{LeGxy2-connected}, $G_1+x^*y^*$ is 2-connected. Note that every vertex in $G_1$ other than $x^*,y^*$ has degree at least $k$. By Theorem \ref{thm: k adm paths}, $G_1$ has $k-1$ admissible $(x^*,y^*)$-paths. It follows that $G$ has $k-1$ admissible $(S,T)$-paths with all internal vertices in $Q$. By Lemma~\ref{lemma: s-con-odd-paths}, $F$ has two $(x,y)$-paths of lengths 1 and 3 for any $x\in S$, $y\in T$. By Lemma \ref{lemma: concatenating paths}, $G$ has $k$ admissible cycles, a contradiction. 
    
    Now assume that $N_S(Q)=\emptyset$, i.e., $N_F(Q)\subseteq T$. It follows that $t\geqslant3$ since $G$ is 3-connected. Thus, $G$ has two disjoint edges from $Q$ to $T$, say $y_1,y_2$ are two end-vertices in $T$ of such two edges. Let $(T_1,T_2)$ be a bipartition of $T$ with $y_1\in T_1$, $y_2\in T_2$. Let $G_2$ be the graph obtained from $Q$ by adding two vertices $y_1^*$ and $y_2^*$, and for each vertex $v\in V(Q)$, 
    \begin{itemize}
	      \item[(1)] adding an edge $vy_1^*$ if $v$ has a neighbor in $T_1$, and
        \item[(2)] adding an edge $vy_2^*$ if $v$ has a neighbor in $T_2$.
    \end{itemize}
    Clearly, $G_2$ has two disjoint edges from $Q$ to $\{y_1^*,y_2^*\}$, and each end block of $Q$ has an inner-vertex adjacent to $y_1^*$ or $y_2^*$. By Lemma \ref{LeGxy2-connected}, $G_2+y_1^*y_2^*$ is 2-connected. Note that every vertex in $G_2$ other than $y_1^*,y_2^*$ has degree at least $k$. By Theorem \ref{thm: k adm paths}, $G_2$ has $k-1$ admissible $(y_1^*,y_2^*)$-paths, implying that $G$ has $k-1$ admissible $(T_1,T_2)$-paths with all internal vertices in $Q$. Recall that $t\geqslant3$, $F$ has two $(y_1,y_2)$-paths of lengths 2 and 4 for any $y_1\in T_1$ and $y_2\in T_2$. By Lemma \ref{lemma: concatenating paths}, $G$ has $k$ admissible cycles, a contradiction. 
\end{proof}

    Recall that $Q$ is adjacent to $S$ or $T$. Now we show that $Q$ is adjacent to only one of $S$, $T$.

\begin{claim}\label{ClNSQorNTQEmpty}
    Either $N_S(Q)=\emptyset$ or $N_T(Q)=\emptyset$.
\end{claim}

\begin{proof}
    Assume the opposite that $Q$ is adjacent to both $S$ and $T$ (see Figure \ref{Fig: STQG'} (i)). If $F\cong K_{s,t}$, then we define $G_1$ as in the proof of Claim \ref{claim: at least 2 neighbors in H}. Similarly to the analysis in the proof of Claim \ref{claim: at least 2 neighbors in H}, we have that $G_1+x^*y^*$ is 2-connected. Now every vertex of $G_1$ other than $x^*,y^*$ has degree at least $k-r+2$; implying that $G_1$ has $k-r+1$ admissible $(x^*,y^*)$-paths. It follows that $G$ has $k-r+1$ admissible $(S,T)$-paths with all internal vertices in $Q$. By Lemma~\ref{lemma: s-con-odd-paths}, $F$ has $(x,y)$-paths of consecutive odd lengths from 1 to $2r-1$ for any $x\in S$, $y\in T$. By Lemma~\ref{lemma: concatenating paths}, $G$ has $k$ admissible cycles, a contradiction. 
    
    So assume that $F\cong K_{s,t}^-$, say $x_1y_1\notin E(F)$, where $x_1\in S$ and $y_1\in T$. Recall that $G$ is 3-connected, implying that $N_F(Q)\backslash\{x_1,y_1\}\neq\emptyset$. We can assume w.l.o.g. that $N_T(Q)\backslash\{y_1\}\neq\emptyset$. Let $G_3$ be the graph obtained from $Q$ by adding two vertices $x^*$ and $y^*$, and for each vertex $v\in V(Q)$,
    \begin{itemize}
	    \item[(1)] adding an edge $vx^*$ if $v$ has a neighbor in $S$, and
        \item[(2)] adding an edge $vy^*$ if $v$ has a neighbor in $T\backslash\{y_1\}$. 
    \end{itemize}
    Since $N_S(Q)\neq\emptyset$, $N_T(Q)\backslash\{y_1\}\neq\emptyset$ and $G$ is triangle-free, we see that $Q$ has two disjoint edges from $Q$ to $S$ and $T\backslash\{y_1\}$, respectively. By $G$ being 3-connected, every end block of $Q$ has an inner-vertex adjacent to $V(F)\backslash\{y_1\}$. By Lemma \ref{LeGxy2-connected}, $G_3+x^*y^*$ is 2-connected.  Note that every vertex in $G_3$ other than $x^*,y^*$ has degree at least $\min \{k-1,k-r+2\}=k-r+2$ (since $r\geq 3$). By Theorem \ref{thm: k adm paths}, $G_3$ has $k-r+1$ admissible $(x^*,y^*)$-paths, implying that $G$ has $k-r+1$ admissible $(S,T\backslash\{y_1\})$-paths with all internal vertices in $Q$. By Lemma \ref{lemma: s-con-odd-paths}, $F$ has $(x,y)$-paths of consecutive odd lengths from 1 to $2s-1$ for any $x\in S$, $y\in T\backslash\{y_1\}$. By Lemma \ref{lemma: concatenating paths}, $G$ has $k$ admissible cycles, a contradiction. 
\end{proof}

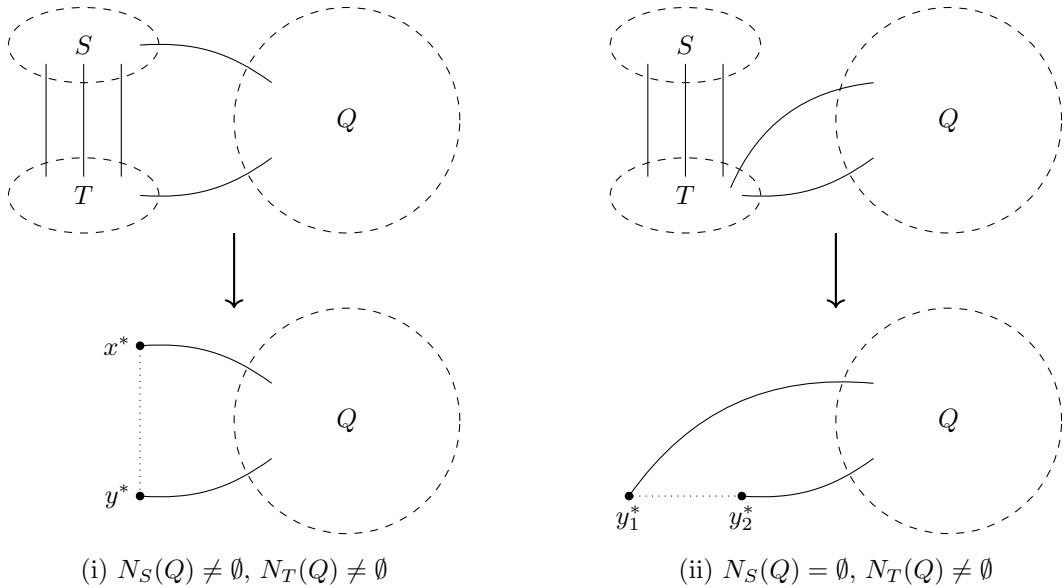
\begin{figure}[ht]
\centering
\begin{tikzpicture}[scale=0.5]

\begin{scope}[xshift=-8cm]
\draw[dashed] (-4,-2) {node {$T$}} ellipse [x radius=2, y radius=1];
\draw[dashed] (-4,2) {node {$S$}} ellipse [x radius=2, y radius=1];
\draw[dashed] (3,0) {node {$Q$}} circle (3);
\foreach \x in {-5,-4,-3} \draw (\x,-1.5)--(\x,1.5);
\draw (-2.5,-2) [bend right=20] to (1,-1);
\draw (-2.5,2) [bend left=20] to (1,1);
\end{scope}

\begin{scope}[xshift=8cm]
\draw[dashed] (-4,-2) {node {$T$}} ellipse [x radius=2, y radius=1];
\draw[dashed] (-4,2) {node {$S$}} ellipse [x radius=2, y radius=1];
\draw[dashed] (3,0) {node {$Q$}} circle (3);
\foreach \x in {-5,-4,-3} \draw (\x,-1.5)--(\x,1.5);
\draw (-2.5,-2) [bend right=20] to (1,-1);
\draw (-2.8,-1.8) [bend left=30] to (1,1);
\end{scope}

\begin{scope}[xshift=-8cm, yshift=-8cm]
\draw[fill=black] (-2.5,-2) {node[left] {$y^*$}} circle (0.1);
\draw[fill=black] (-2.5,2) {node[left] {$x^*$}} circle (0.1);
\draw[dashed] (3,0) {node {$Q$}} circle (3);
\draw (-2.5,-2) [bend right=20] to (1,-1);
\draw (-2.5,2) [bend left=20] to (1,1);
\draw[dotted] (-2.5,-2)--(-2.5,2);
\draw[->, thick] (0,5)--(0,3);
\node at (0,-4) {(i) $N_S(Q)\neq\emptyset$, $N_T(Q)\neq\emptyset$};
\end{scope}

\begin{scope}[xshift=8cm, yshift=-8cm]
\draw[fill=black] (-2.5,-2) {node[below] {$y_2^*$}} circle (0.1);
\draw[fill=black] (-5.5,-2) {node[below] {$y_1^*$}} circle (0.1);
\draw[dashed] (3,0) {node {$Q$}} circle (3);
\draw (-2.5,-2) [bend right=20] to (1,-1);
\draw (-5.5,-2) [bend left=30] to (1,1);
\draw[dotted] (-2.5,-2)--(-5.5,-2);
\draw[->, thick] (0,5)--(0,3);
\node at (0,-4) {(ii) $N_S(Q)=\emptyset$, $N_T(Q)\neq\emptyset$};
\end{scope}

\end{tikzpicture}
\caption{Constructions of $S$, $T$, $Q$ and $G_i$ for $i\in \{1,2,3,4\}$ in Lemma \ref{lemma: 2-con-bi-H}.}\label{Fig: STQG'}
\end{figure}

    By Claim \ref{ClNSQorNTQEmpty}, we can assume w.l.o.g. that $N_S(Q)=\emptyset$ (see Figure \ref{Fig: STQG'} (ii)). By $G$ being 2-connected, $G$ has two disjoint edges from $Q$ to $T$, say $y_1,y_2$ are the end-vertices of such two edges in $T$. Let $(T_1,T_2)$ be a bipartition of $T$ with $y_1\in T_1$, $y_2\in T_2$. 

    Let $G_4$ be the graph obtained from $Q$ by adding two vertices $y_1^*$ and $y_2^*$, and for each vertex $v\in V(Q)$, 
    \begin{itemize}
	    \item[(1)] adding an edge $vy_1^*$ if $d_{T_1}(v)\geq 1$ or $d_T(v)\geq 2$, and
        \item[(2)] adding an edge $vy_2^*$ if $d_{T_2}(v)\geq 1$ or $d_T(v)\geq 2$.
    \end{itemize}
    We remark that if $v\in V(Q)$ has at least two neighbors in $T$, then both $vy_1^*,vy_2^*\in E(G_4)$. Recall that $G$ has two disjoint edges from $Q$ to $T_1,T_2$, respectively, implying that $G_4$ has two disjoint edges from $Q$ to $\{y_1^*,y_2^*\}$. Also note that every end block of $Q$ has an inner-vertex adjacent to $T$ in $G$, implying that every end block of $Q$ has an inner-vertex adjacent to $y_1^*$ or $y_2^*$ in $G_4$. By Lemma \ref{LeGxy2-connected}, $G_4+y_1^*y_2^*$ is 2-connected. Note that every vertex in $G_4$ other than $y_1^*,y_2^*$ has degree at least $\min\{k,k-(r-1)+2\}=k-r+3$. By Theorem \ref{thm: k adm paths}, $G_4$ has $k-r+2$ admissible $(y_1^*,y_2^*)$-paths, implying that $G$ has $k-r+2$ admissible $T$-paths with all internal vertices in $Q$. By Lemma \ref{lemma: s-con-odd-paths}, $F$ has $(y'_1,y'_2)$-paths of consecutive even lengths from 2 to $2r-2$ for each two distinct vertices $y'_1,y'_2\in T$. By Lemma \ref{lemma: concatenating paths X}, $G$ has $k$ admissible cycles, a contradiction. 
\end{proof}

\begin{lemma}\label{lemma: 3-con-grith4}
    Let $G$ be a $2$-connected graph with minimum degree at least $k\geqslant 3$. If $G$ contains a $4$-cycle, then $G$ contains $k$ admissible cycles, except that $G\cong K_{k+1}$ or $K_{k,n-k}$.
\end{lemma}

\begin{proof}
    If $G$ contains a triangle, then, by Theorem \ref{thm: 2-con-K3-k-cycles}, $G$ contains $k$ admissible cycles or $G\cong K_{k+1}$, as desired. So assume that $G$ does not contain a triangle. Assume also that $G$ does not contain $k$ admissible cycles. Let $s$ be the maximum integer such that $G$ contains a $K_{s,s}$. Then $s\geqslant 2$ and $G$ does not contain a $K_{s+1,s+1}$. 
    
    Suppose first that $G$ contains a subgraph $F \cong K_{s+1,s+1}^-$ with bipartition $(S,T)$. By the absence of $K_{s+1,s+1}$, every vertex in $V(G)\backslash V(F)$ has at most $s$ neighbors in $S$ and at most $s$ neighbors in $T$. If a vertex in $V(G)\backslash V(F)$ has neighbors in both $S$ and $T$, then it has exactly $2\leqslant s$ neighbors in $V(F)$ since $G$ is triangle-free. This implies that every vertex in $V(G)\backslash V(F)$ has at most $s$ neighbors in $V(F)$. By Lemma \ref{lemma: 2-con-bi-H}, $G$ contains $k$ admissible cycles or $G$ is isomorphic to $K_{k+1}$ or $K_{k,n-k}$, a contradiction. So $G$ does not contain a $K_{s+1,s+1}^-$. 
    
    Suppose now that $G$ contains a $K_{s,s+1}$. Let $F \cong K_{s,t}$ be a subgraph of $G$ with bipartition $(S,T)$ having maximum $t=|T|$, where $t\geqslant s+1$. By $G$ being triangle-free, every vertex in $V(G)\backslash V(F)$ cannot have neighbors in both $S$ and $T$. By the maximality of $t$, every vertex in $V(G)\backslash V(F)$ has at most $s-1$ neighbors in $S$. If a vertex in $V(G)\backslash V(F)$ has $s$ neighbors in $T$, then $G$ contains a $K_{s+1,s+1}^-$, a contradiction. So every vertex in $V(G)\backslash V(F)$ has at most $s-1$ neighbors in $T$. It follows that every vertex in $V(G)\backslash V(F)$ has at most $s-1$ neighbors in $V(F)$. By Lemma \ref{lemma: 2-con-bi-H}, $G$ contains $k$ admissible cycles or $G$ is isomorphic to $K_{k+1}$ or $K_{k,n-k}$, a contradiction. So $G$ does not contain a $K_{s,s+1}$. 
    
    Finally, let $F \cong K_{s,s}$ be a subgraph of $G$ with bipartition $(S,T)$. By $G$ being triangle-free, every vertex in $V(G)\backslash V(F)$ cannot have neighbors in both $S$ and $T$. If a vertex in $V(G)\backslash V(F)$ has $s$ neighbors in $S$ (or in $T$), then $G$ contains a $K_{s,s+1}$, a contradiction. This implies that every vertex in $V(G)\backslash V(F)$ has at most $s-1$ neighbors in $V(F)$. By Lemma \ref{lemma: 2-con-bi-H}, $G$ contains $k$ admissible cycles or $G$ is isomorphic to $K_{k+1}$ or $K_{k,n-k}$, a contradiction.
\end{proof}

By Theorem~\ref{thm: 2-con-K3-k-cycles} and Lemma~\ref{lemma: 2-con-bi-H} in order to prove Theorem~\ref{thm: k adm cycles min degree k}, we may assume that $G$ contains neither a triangle nor a 4-cycle. Moreover, by Lemma~\ref{lemma: high connectivity no k adm cycles} the connectivity of $G$ is at least $(k+2)/2 \geq 3$. Therefore, the following results imply that for $k\geqslant 6$ Theorem~\ref{thm: k adm cycles min degree k} holds.

\begin{theorem}[Gao, Huo, Liu and Ma \cite{G+22}]\label{thm: 3-con-bi}
    Let $k\geqslant 5$ be an integer and $G$ a $3$-connected bipartite graph of minimum degree at least $k$. If $G$ contains no $4$-cycles, then $G$ contains $k$ admissible cycles.
\end{theorem}

\begin{theorem}[Gao, Huo, Liu and Ma \cite{G+22}]\label{thm: 3-con-noK3-k-cycles}
    Let $G$ be a $3$-connected non-bipartite graph with minimum degree at least $k\geqslant 6$. If $G$ does not contain a triangle, then $G$ contains $k$ admissible cycles. Furthermore, if $k\geqslant 8$, then $G$ contains $k$ cycles of consecutive lengths; and if $k\in \{6, 7\}$, then $G$ contains cycles of all lengths modulo $k$.
\end{theorem}

To prove Theorem~\ref{thm: k adm cycles min degree k} in the remaining case $k \in \{4,5\}$, we use the following results.

\begin{lemma}\label{Le3ConnectedContract}
    Let $G$ be a $3$-connected graph, $C$ a cycle of $G$, and $Q$ a component of $G-V(C)$. If $G'$ is the graph obtained from $G$ by contracting all edges in $Q$, then $G'$ is $3$-connected.
\end{lemma}

\begin{proof}
    Denote by $w$ the vertex of $G'$ obtained by contracting $Q$. Suppose that $\{x,y\}$ is a 2-cut of $G'$. If $w\notin\{x,y\}$, then $\{x,y\}$ is also a 2-cut of $G$, a contradiction. So we have that $w\in\{x,y\}$. It follows that $G-Q=G'-w$ has a cut-vertex, a contradiction.
\end{proof}

\begin{theorem}[Gao, Li, Ma and Xie \cite{GLMX}]\label{thm: adm paths degree sum}
    Let $G$ be a graph with $x,y\in V(G)$ such that $G+xy$ is $2$-connected. If every vertex of $G$ other than $x$ and $y$ has degree at least $k$, and any edge $uv\in E(G)$ with $\{u,v\}\cap \{x, y\}=\emptyset$ has degree sum $d(u)+d(v)\geqslant 2k+1$, then $G$ has $k$ admissible $(x,y)$-paths of consecutive lengths or $k-1$ admissible $(x,y)$-paths whose lengths form an arithmetic progression with common difference two.
\end{theorem}

\begin{lemma}\label{lemma: concatenating paths a1a2k}
    Let $k\in\{4,5\}$. Let $G$ be a graph, $H$ a subgraph of $G$ and $X,Y$ two disjoint subsets of $V(H)$. Assume that $a_1,a_2$ are two integers with $a_2-a_1=1$ or $4$, such that for every $x\in X$, $y\in Y$, $i\in \{1,2\}$, $H$ has an $(x,y)$-path of length $a_i$. Assume also that $G$ has $k-1$ admissible $(X,Y)$-paths of consecutive lengths or $k-2$ admissible $(X,Y)$-paths whose lengths form an arithmetic progression with common difference two, with all internal vertices in $V(G)\backslash V(H)$. Then $G$ has $k$ admissible cycles.
\end{lemma}

\begin{proof}
    If $a_2-a_1=1$, then the assertion is deduced by Lemma \ref{lemma: concatenating paths}. Now assume that $a_2-a_1=4$. Consider an arbitrary $(X,Y)$-path $P$ with all internal vertices in $V(G)\backslash V(H)$, say of length $p$ and with end-vertices $x\in X, y\in Y$. Together with two $(x,y)$-paths in $H$ of lengths $a_1$ and $a_2=a_1+4$, respectively, we can get two cycles of $G$ of lengths $p+a_1$ and $p+a_1+4$. 

    Suppose first that $G$ has $k-1$ admissible $(X,Y)$-paths of consecutive lengths, say of lengths $p+i$ for $i=0,\ldots,k-2$. If $k=4$, then $G$ has cycles of lengths $p+a_1, p+a_1+2, p+a_1+4, p+a_1+6$. If $k=5$, then $G$ has cycles of lengths $p+a_1, p+a_1+1, p+a_1+2, p+a_1+3, p+a_1+4$. Suppose now that $G$ has $k-2$ admissible $(X,Y)$-paths whose lengths form an arithmetic progression with common difference two, say of lengths $p+2i$, $i=0,\ldots,k-3$. If $k=4$, then $G$ has cycles of lengths $p+a_1, p+a_1+2, p+a_1+4, p+a_1+6$. If $k=5$, then $G$ has cycles of lengths $p+a_1, p+a_1+2, p+a_1+4, p+a_1+6, p+a_1+8$. In each case, $G$ has $k$ admissible cycles, as desired.       
\end{proof}

Let $C$ be a cycle. If $C$ is even, then we say two vertices $x,y \in V(C)$ are {\em diagonal} if $d_C(x,y)=|C|/2$; {\em quasi-diagonal} if $d_C(x,y)=|C|/2-1$; and {\em sub-quasi-diagonal} if $d_C(x,y)=|C|/2-2$. While if $C$ is odd, then we say $x,y\in V(C)$ are {\em quasi-diagonal} if $d_C(x,y)=(|C|-1)/2$. Note that for any vertex $x\in V(C)$, $|C|\geqslant 5$, $x$ has exactly two quasi-diagonal vertices; and if $|C|$ is even, then $x$ has exactly two sub-quasi-diagonal vertices. 

\begin{lemma}\label{lemma: k adm cycles for k=4,5}
    Let $G$ be a $2$-connected graph with minimum degree at least $k$, where $k\in \{4,5\}$. Then $G$ contains $k$ admissible cycles, unless $G\cong K_{k+1}$ or $K_{k,n-k}$.
\end{lemma}

\begin{proof}
   We prove by contradiction. Let $G$ be a 2-connected graph with minimum degree at least $k$, that does not contain $k$ admissible cycles, and is not isomorphic to neither $K_{k+1}$ nor $K_{k,n-k}$. 
   By Lemma \ref{lemma: high connectivity no k adm cycles}, $G$ is $3$-connected.
   By Theorem \ref{thm: 2-con-K3-k-cycles} and Lemma \ref{lemma: 3-con-grith4}, we can assume that $G$ does not contain a triangle or a 4-cycle. We now show that $G$ does not contain a 5-cycle.

\setcounter{claim}{0}
\begin{claim}\label{ClNoC5}
    $G$ contains no $5$-cycle.
\end{claim}

\begin{proof}   
    Assume the opposite that $G$ contains a 5-cycle. Define $F_r$ to be the graph obtained by joining two vertices $x_1,x_2$ with $r$ paths of length 3 and one path of length 2 such that all the paths are pairwise internally disjoint, see Figure \ref{FiBr} for an illustration. Note that $F_1$ is a 5-cycle.
    
\begin{figure}[ht]
\centering
\begin{tikzpicture}[scale=0.5]
\draw[fill=black] (0,0) {coordinate (x)} {node[right] {$x$}} circle (0.1);
\draw[fill=black] (0,3) {coordinate (u)} {node[above] {$x_1$}} circle (0.1);
\draw[fill=black] (0,-3) {coordinate (v)} {node[below] {$x_2$}} circle (0.1);
\draw[thick] (u)--(v);
\foreach \b in {-1,1} \foreach \c in {-1,1}
{\draw[fill=black] (\b*3,\c*1.5) circle (0.05);
\draw[fill=black] (\b*3-0.5,\c*1.5) circle (0.05);
\draw[fill=black] (\b*3+0.5,\c*1.5) circle (0.05);}
\foreach \a in {1,2,4,5}
{\draw[fill=black] (\a*2-6,-1.5) {coordinate (z\a)} circle (0.1);
\draw[fill=black] (\a*2-6,1.5) {coordinate (y\a)} circle (0.1);
\draw[thick] (u)--(y\a)--(z\a)--(v);}
\node[left] at (y1) {$y_1$}; \node[left] at (z1) {$z_1$};
\node[right] at (y5) {$y_r$}; \node[right] at (z5) {$z_r$};
\end{tikzpicture}
\caption{Graph $F_r$.}\label{FiBr}
\end{figure}
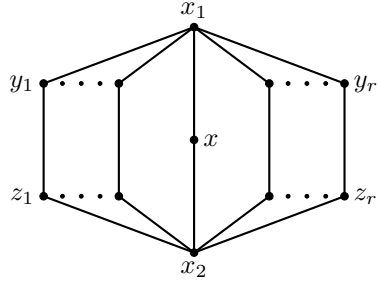

    \begin{subclaim}\label{ClInducedOneNeighbor}
        If $C$ is a $5$-cycle of $G$, then 
        \begin{itemize}
            \item[(1)] $C$ is an induced cycle of $G$,
            \item[(2)] every vertex outside $C$ has at most one neighbor in $C$, and
            \item[(3)] if $P$ is a $V(C)$-path in $G$ of length $3$, then its two end-vertices have distance exactly $2$ in $C$ (that is, $G[V(C)\cup V(P)]\cong F_2$).
        \end{itemize}
    \end{subclaim}

    \begin{proof}
        All three statements can be deduced from the fact that $G$ has neither a triangle nor a 4-cycle.
    \end{proof}
    
    \begin{subclaim}\label{ClNoB3}
        $G$ contains no $F_3$.
    \end{subclaim}

    \begin{proof}
        Suppose that $F\cong F_r$ is a subgraph of $G$ with $r$ maximum and $r\geqslant 3$. We label the vertices of $F$ as in Figure \ref{FiBr}. By Claim \ref{ClInducedOneNeighbor}, $F$ is an induced subgraph of $G$ and any vertex outside $F$ can have at most one neighbor in any 5-cycle in $F$. We further claim that any vertex outside $F$ can have at most one neighbor in $F$. If $v\in V(G)\backslash V(F)$ is adjacent to both $y_1$ and $y_2$, then a 4-cycle appears; if $v$ is adjacent to both $y_1$ and $z_2$, then we can get a contradiction by finding five cycles of lengths from 5 to 9, namely,
        \[
        x_1y_1z_1x_2xx_1,~
        x_1y_1vz_2x_2xx_1,~ 
        x_1y_1vz_2x_2z_3y_3x_1,~
        x_1y_2z_2vy_1z_1x_2xx_1,~
        x_1y_2z_2vy_1z_1x_2z_3y_3x_1,
        \]
        a contradiction. The other cases can be proved similarly. So we conclude that every vertex outside $F$ has at most one neighbor in $F$.

        Set $T=\{y_i,z_i: 1\leqslant i\leqslant r\}$. Recall that $G$ is 3-connected. There is a path from $x$ to $T$ avoiding $x_1,x_2$. It follows that $G-V(F)$ has a component $Q$ such that $x\in N(Q)$ and $N_T(Q)\neq\emptyset$. 

        We claim that every end block of $Q$ has an inner-vertex adjacent to $\{x\}\cup T$. Suppose otherwise that $B$ is an end block of $Q$, $c$ is the cut-vertex of $Q$ contained in $B$, and $B-c$ is not adjacent to $\{x\}\cup T$. Since $G$ is $3$-connected, $N_F(B-c)=\{x_1,x_2\}$. Recall that every vertex outside $F$ has at most one neighbor in $F$. There are two disjoint edges from $B-c$ to $\{x_1,x_2\}$. Let $G_0=G[V(B)\cup\{x_1,x_2\}]$. It follows that $G_0+x_1x_2$ is 2-connected and every vertex in $B-c$ has degree at least $k$ in $G_0$. By Theorem \ref{thm: k adm paths 2}, $G_0$ has $k-1$ admissible $(x_1,x_2)$-paths. Note that $F$ has two $(x_1,x_2)$-paths of lengths 2 and 3, respectively. By Lemma \ref{lemma: concatenating paths}, $G$ has $k$ admissible cycles, a contradiction. Now we conclude that every end block of $Q$ has an inner-vertex adjacent to $\{x\}\cup T$.    
        
        Let $G_1$ be the graph obtained from $G[V(Q)\cup\{x\}]$ by adding a vertex $x^*$ and all edges in $\{vx^*: v\in N_Q(T)\}$. By Lemma \ref{LeGxy2-connected}, $G_1+xx^*$ is 2-connected. Recall that every vertex in $Q$ has at most one neighbor in $F$. If there is an edge $v_1v_2$ such that $v_1x_1,v_2x_2\in E(G)$, then $F\cup x_1v_1v_2x_2$ is a $F_{r+1}$, contradicting the choice of $F$. It follows that every vertex in $Q$ has degree at least $k-1$ in $G_1$, and every pair of adjacent vertices in $Q$ has degree sum at least $2k-1$ in $G_1$. 
        
        By Theorem \ref{thm: adm paths degree sum}, $G_1$ has $k-1$ admissible $(x,x^*)$-paths of consecutive lengths or $k-2$ admissible $(x,x^*)$-paths whose lengths form an arithmetic progression with common difference 2. This implies that $G$ has $k-1$ admissible $(x,T)$-paths of consecutive lengths or $k-2$ admissible $(x,T)$-paths whose lengths form an arithmetic progression with common difference 2, with all internal vertices in $Q$. Note that $F$ contains two $(x,x')$-paths of lengths 2 and 3 for each $x'\in T$. By Lemma \ref{lemma: concatenating paths a1a2k}, $G$ contains $k$ admissible cycles, a contradiction.
    \end{proof}
    
    \begin{subclaim}\label{ClNoPetersen}
        $G$ contains no Petersen graph.
    \end{subclaim}

    \begin{proof}
        Let $F$ be a subgraph of $G$ isomorphic to the Petersen graph and let $Q$ be a component of $G-V(F)$. Since $G$ is $3$-connected, $N_F(Q)$ has at least three vertices, and two of them are nonadjacent. Thus, there exists $x\in N_F(Q)$ such that for $T:=V(F)\backslash(\{x\}\cup N_F(x))$ it holds $N_T(Q)\neq\emptyset$.

        Note that each two vertices of $F$ are contained in a common 5-cycle. By Claim \ref{ClInducedOneNeighbor}, every vertex outside $F$ has at most one neighbor in $F$. Now we claim that for each two adjacent vertices outside $F$, at least one of them has no neighbor in $F$. Suppose that $v_1v_2\in E(Q)$ such that $v_1x_1,v_2x_2\in E(G)$ with $x_1,x_2\in V(F)$. Since $G$ has neither a triangle nor a 4-cycle, $x_1,x_2$ are at distance 2 in $F$. Then $F\cup x_1v_1v_2x_2$ contains a $F_3$, contradicting Claim \ref{ClNoB3}. So we conclude that for each two adjacent vertices outside $F$, at least one of them has no neighbor in~$F$.

        Now we claim that every end block of $Q$ has an inner-vertex adjacent to $\{x\}\cup T$. Suppose that $B$ is an end block of $Q$, $c$ is the cut-vertex of $Q$ contained in $B$, such that $B-c$ is not adjacent to $\{x\}\cup T$. Then $N_F(B-c)\subseteq N_F(x)$. Set $N_F(x)=\{y_1,y_2,y_3\}$ such that $y_1\in N_F(B-c)$. Let $G_0$ the graph obtained from $G[V(B)\cup\{y_1\}]$ by adding a vertex $y^*$ and all edges in $\{vy^*: v\in N_B(y_2)\cup N_B(y_3)\}$. It follows that $G_0+y_1y^*$ is 2-connected and every vertex in $B-c$ has degree at least $k$ in $G_0$. By Theorem \ref{thm: k adm paths 2}, $G_0$ has $k-1$ admissible $(y_1,y^*)$-paths, implying that $G$ has $k-1$ admissible $(y_1,\{y_2,y_3\})$-paths with all internal vertices in $B$. Note that $F$ contains two $(y_1,y_i)$-paths of lengths 2 and 3, respectively, for each $i=2,3$. By Lemma \ref{lemma: concatenating paths}, $G$ has $k$ admissible cycles, a contradiction. Now we conclude that every end block of $Q$ has an inner-vertex adjacent to $\{x\}\cup T$. 
        
        Let $G_1$ be the graph obtained from $G[V(Q)\cup\{x\}]$ by adding a vertex $x^*$ and all edges in $\{vx^*: v\in N_Q(T)\}$. Recall that $G$ does not contain a 4-cycle and $N_Q(T)\neq \emptyset$. This implies the existence of two disjoint edges from $\{x,x^*\}$ to $Q$ in $G_1$. By Lemma \ref{LeGxy2-connected}, $G_1+xx^*$ is 2-connected. Recall that every vertex in $Q$ has at most one neighbor in $F$, and that for every two adjacent vertices in $Q$, at most one of them has a neighbor in $F$. It follows that every vertex in $Q$ has degree at least $k-1$ in $G_1$, and every pair of adjacent vertices in $Q$ has degree sum at least $2k-1$ in $G_1$. By Theorem \ref{thm: adm paths degree sum}, $G_1$ has $k-1$ admissible $(x,x^*)$-paths of consecutive lengths or $k-2$ admissible $(x,x^*)$-paths whose lengths form an arithmetic progression with common difference 2. It follows that $G$ has $k-1$ admissible $(x,T)$-paths of consecutive lengths or $k-2$ admissible $(x,T)$-paths whose lengths form an arithmetic progression with common difference 2, with all internal vertices in $Q$. Note that $F$ contains two $(x,x')$-paths of lengths 2 and 3 for each $x'\in T$. By Lemma \ref{lemma: concatenating paths a1a2k}, $G$ contains $k$ admissible cycles, a contradiction.         
    \end{proof}

    Now we let $C$ be a 5-cycle of $G$, say $C=x_1x_2x_3x_4x_5x_1$. 

    \begin{subclaim}\label{Clx1y1z3x3}
        There are two vertices $y_5,z_2\in V(G)\backslash V(C)$ such that $x_5y_5,y_5z_2,z_2x_2\in E(G)$. 
    \end{subclaim}

    \begin{proof}
        Suppose otherwise. Since $G$ is 3-connected, there is a component $Q$ of $G-V(C)$ such that $x_1\in N_C(Q)$ and $\{x_3,x_4\}\cap N_C(Q)\neq\emptyset$. Let $G_1$ be the graph obtained from $G[V(Q)\cup\{x_1\}]$ by adding a vertex $x^*$ and all edges in $\{vx^*: v\in N_Q(\{x_3,x_4\})\}$. Similarly to the proof in Claim \ref{ClNoPetersen}, we see that $G_1+x_1x^*$ is 2-connected. Recall that every vertex in $Q$ has at most one neighbor in $C$, and at least one of each pair of adjacent vertices in $Q$ is not adjacent to $\{x_2,x_5\}$. It follows that every vertex in $Q$ has degree at least $k-1$ in $G_1$, and every pair of adjacent vertices in $Q$ has degree sum at least $2k-1$ in $G_1$. By Theorem \ref{thm: adm paths degree sum}, $G_1$ has $k-1$ admissible $(x_1,x^*)$-paths of consecutive lengths or $k-2$ admissible $(x_1,x^*)$-paths whose lengths form an arithmetic progression with common difference 2. It follows that $G$ has $k-1$ admissible $(x_1,\{x_3,x_4\})$-paths of consecutive lengths or $k-2$ admissible $(x_1,\{x_3,x_4\})$-paths whose lengths form an arithmetic progression with common difference two, with all internal vertices in $Q$. Note that $C$ has two $(x_1,x')$-paths of lengths 2 and 3 for each $x'\in\{x_3,x_4\}$. By Lemma \ref{lemma: concatenating paths a1a2k}, $G$ contains $k$ admissible cycles, a contradiction.         
    \end{proof}

    By symmetry and Claim \ref{Clx1y1z3x3}, we obtain that for each $i=1,\ldots,5$, there are two vertices $y_i,z_{i+2}\in V(G)\backslash V(C)$ such that $x_iy_i,y_iz_{i+2},z_{i+2}x_{i+2}\in E(G)$, where the indices are taken modulo 5. We remark that possibly $y_i=z_i$ (and for distinct $i,j$, $\{y_i,z_i\}\cap\{y_j,z_j\}=\emptyset$). 
    
    We define the hypo-Petersen graphs as graphs depicted in Figure \ref{FiHypoPetersen}. We say a hypo-Petersen graph is a proper one if it is not a Petersen graph. We remark that every proper hypo-Petersen graph contains cycles of consecutive lengths from 5 to 9, see Figure \ref{FiCycle6789Hypo}.
    
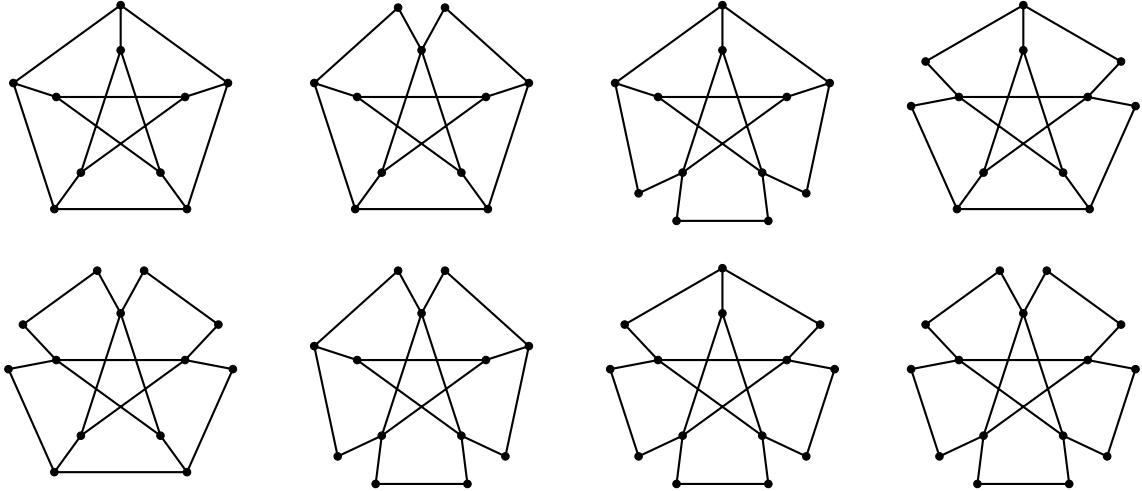
\begin{figure}[ht]
\centering
\begin{tikzpicture}[scale=0.5]

\begin{scope}[xshift=-12cm]
\foreach \x in {1,2,...,5} \draw[fill=black] (\x*72+18:1.8) {coordinate (x\x)} circle (0.1);
\draw[thick] (x1)--(x3)--(x5)--(x2)--(x4)--(x1);
\foreach \y in {1,2,3,4,5} 
{\draw[fill=black] (\y*72+18:3) {coordinate (y\y)} {coordinate (z\y)} circle (0.1);
\draw[thick] (x\y)--(y\y);}
\draw[thick] (z1)--(y2) (z2)--(y3) (z3)--(y4) (z4)--(y5) (z5)--(y1);
\end{scope}

\begin{scope}[xshift=-4cm]
\foreach \x in {1,2,...,5} \draw[fill=black] (\x*72+18:1.8) {coordinate (x\x)} circle (0.1);
\draw[thick] (x1)--(x3)--(x5)--(x2)--(x4)--(x1);
\foreach \y in {1} 
{\draw[fill=black] (\y*72+6:3) {coordinate (y\y)} circle (0.1);
\draw[fill=black] (\y*72+30:3) {coordinate (z\y)} circle (0.1); 
\draw[thick] (x\y)--(y\y) (x\y)--(z\y);}
\foreach \y in {2,3,4,5} 
{\draw[fill=black] (\y*72+18:3) {coordinate (y\y)} {coordinate (z\y)} circle (0.1);
\draw[thick] (x\y)--(y\y);}
\draw[thick] (z1)--(y2) (z2)--(y3) (z3)--(y4) (z4)--(y5) (z5)--(y1);
\end{scope}

\begin{scope}[xshift=4cm]
\foreach \x in {1,2,...,5} \draw[fill=black] (\x*72+18:1.8) {coordinate (x\x)} circle (0.1);
\draw[thick] (x1)--(x3)--(x5)--(x2)--(x4)--(x1);
\foreach \y in {3,4} 
{\draw[fill=black] (\y*72+6:3) {coordinate (y\y)} circle (0.1);
\draw[fill=black] (\y*72+30:3) {coordinate (z\y)} circle (0.1); 
\draw[thick] (x\y)--(y\y) (x\y)--(z\y);}
\foreach \y in {1,2,5} 
{\draw[fill=black] (\y*72+18:3) {coordinate (y\y)} {coordinate (z\y)} circle (0.1);
\draw[thick] (x\y)--(y\y);}
\draw[thick] (z1)--(y2) (z2)--(y3) (z3)--(y4) (z4)--(y5) (z5)--(y1);
\end{scope}

\begin{scope}[xshift=12cm]
\foreach \x in {1,2,...,5} \draw[fill=black] (\x*72+18:1.8) {coordinate (x\x)} circle (0.1);
\draw[thick] (x1)--(x3)--(x5)--(x2)--(x4)--(x1);
\foreach \y in {2,5} 
{\draw[fill=black] (\y*72+6:3) {coordinate (y\y)} circle (0.1);
\draw[fill=black] (\y*72+30:3) {coordinate (z\y)} circle (0.1); 
\draw[thick] (x\y)--(y\y) (x\y)--(z\y);}
\foreach \y in {1,3,4} 
{\draw[fill=black] (\y*72+18:3) {coordinate (y\y)} {coordinate (z\y)} circle (0.1);
\draw[thick] (x\y)--(y\y);}
\draw[thick] (z1)--(y2) (z2)--(y3) (z3)--(y4) (z4)--(y5) (z5)--(y1);
\end{scope}

\begin{scope}[xshift=-12cm, yshift=-7cm]
\foreach \x in {1,2,...,5} \draw[fill=black] (\x*72+18:1.8) {coordinate (x\x)} circle (0.1);
\draw[thick] (x1)--(x3)--(x5)--(x2)--(x4)--(x1);
\foreach \y in {1,2,5} 
{\draw[fill=black] (\y*72+6:3) {coordinate (y\y)} circle (0.1);
\draw[fill=black] (\y*72+30:3) {coordinate (z\y)} circle (0.1); 
\draw[thick] (x\y)--(y\y) (x\y)--(z\y);}
\foreach \y in {3,4} 
{\draw[fill=black] (\y*72+18:3) {coordinate (y\y)} {coordinate (z\y)} circle (0.1);
\draw[thick] (x\y)--(y\y);}
\draw[thick] (z1)--(y2) (z2)--(y3) (z3)--(y4) (z4)--(y5) (z5)--(y1);
\end{scope}

\begin{scope}[xshift=-4cm, yshift=-7cm]
\foreach \x in {1,2,...,5} \draw[fill=black] (\x*72+18:1.8) {coordinate (x\x)} circle (0.1);
\draw[thick] (x1)--(x3)--(x5)--(x2)--(x4)--(x1);
\foreach \y in {1,3,4} 
{\draw[fill=black] (\y*72+6:3) {coordinate (y\y)} circle (0.1);
\draw[fill=black] (\y*72+30:3) {coordinate (z\y)} circle (0.1); 
\draw[thick] (x\y)--(y\y) (x\y)--(z\y);}
\foreach \y in {2,5} 
{\draw[fill=black] (\y*72+18:3) {coordinate (y\y)} {coordinate (z\y)} circle (0.1);
\draw[thick] (x\y)--(y\y);}
\draw[thick] (z1)--(y2) (z2)--(y3) (z3)--(y4) (z4)--(y5) (z5)--(y1);
\end{scope}

\begin{scope}[xshift=4cm, yshift=-7cm]
\foreach \x in {1,2,...,5} \draw[fill=black] (\x*72+18:1.8) {coordinate (x\x)} circle (0.1);
\draw[thick] (x1)--(x3)--(x5)--(x2)--(x4)--(x1);
\foreach \y in {2,3,4,5} 
{\draw[fill=black] (\y*72+6:3) {coordinate (y\y)} circle (0.1);
\draw[fill=black] (\y*72+30:3) {coordinate (z\y)} circle (0.1); 
\draw[thick] (x\y)--(y\y) (x\y)--(z\y);}
\foreach \y in {1} 
{\draw[fill=black] (\y*72+18:3) {coordinate (y\y)} {coordinate (z\y)} circle (0.1);
\draw[thick] (x\y)--(y\y);}
\draw[thick] (z1)--(y2) (z2)--(y3) (z3)--(y4) (z4)--(y5) (z5)--(y1);
\end{scope}

\begin{scope}[xshift=12cm, yshift=-7cm]
\foreach \x in {1,2,...,5} \draw[fill=black] (\x*72+18:1.8) {coordinate (x\x)} circle (0.1);
\draw[thick] (x1)--(x3)--(x5)--(x2)--(x4)--(x1);
\foreach \y in {1,2,3,4,5} 
{\draw[fill=black] (\y*72+6:3) {coordinate (y\y)} circle (0.1);
\draw[fill=black] (\y*72+30:3) {coordinate (z\y)} circle (0.1); 
\draw[thick] (x\y)--(y\y) (x\y)--(z\y);}
\draw[thick] (z1)--(y2) (z2)--(y3) (z3)--(y4) (z4)--(y5) (z5)--(y1);
\end{scope}

\end{tikzpicture}
\caption{Hypo-Petersen graphs.}\label{FiHypoPetersen}
\end{figure}

    Now by Claims \ref{Clx1y1z3x3} and \ref{ClNoPetersen}, $H=G[\{x_i,y_i,z_i: i=1,\ldots,5\}]$ is a proper hypo-Petersen graph. It follows that $H$, and then, $G$, has cycles of lengths 5, 6, 7, 8, 9, a contradiction.
\end{proof}
    
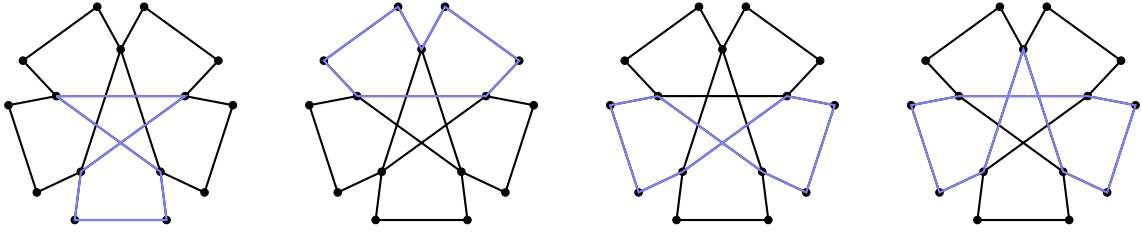
\begin{figure}[ht]
\centering
\begin{tikzpicture}[scale=0.5]

\begin{scope}[xshift=-12cm]
\foreach \x in {1,2,...,5} \draw[fill=black] (\x*72+18:1.8) {coordinate (x\x)} circle (0.1);
\draw[thick] (x1)--(x3)--(x5)--(x2)--(x4)--(x1);
\foreach \y in {1,2,3,4,5} 
{\draw[fill=black] (\y*72+6:3) {coordinate (y\y)} circle (0.1);
\draw[fill=black] (\y*72+30:3) {coordinate (z\y)} circle (0.1); 
\draw[thick] (x\y)--(y\y) (x\y)--(z\y);}
\draw[thick] (z1)--(y2) (z2)--(y3) (z3)--(y4) (z4)--(y5) (z5)--(y1);
\draw[thick, blue!50] (x2)--(x5)--(x3)--(z3)--(y4)--(x4)--(x2);
\end{scope}

\begin{scope}[xshift=-4cm]
\foreach \x in {1,2,...,5} \draw[fill=black] (\x*72+18:1.8) {coordinate (x\x)} circle (0.1);
\draw[thick] (x1)--(x3)--(x5)--(x2)--(x4)--(x1);
\foreach \y in {1,2,3,4,5} 
{\draw[fill=black] (\y*72+6:3) {coordinate (y\y)} circle (0.1);
\draw[fill=black] (\y*72+30:3) {coordinate (z\y)} circle (0.1); 
\draw[thick] (x\y)--(y\y) (x\y)--(z\y);}
\draw[thick] (z1)--(y2) (z2)--(y3) (z3)--(y4) (z4)--(y5) (z5)--(y1);
\draw[thick, blue!50] (x2)--(x5)--(z5)--(y1)--(x1)--(z1)--(y2)--(x2);
\end{scope}

\begin{scope}[xshift=4cm]
\foreach \x in {1,2,...,5} \draw[fill=black] (\x*72+18:1.8) {coordinate (x\x)} circle (0.1);
\draw[thick] (x1)--(x3)--(x5)--(x2)--(x4)--(x1);
\foreach \y in {1,2,3,4,5} 
{\draw[fill=black] (\y*72+6:3) {coordinate (y\y)} circle (0.1);
\draw[fill=black] (\y*72+30:3) {coordinate (z\y)} circle (0.1); 
\draw[thick] (x\y)--(y\y) (x\y)--(z\y);}
\draw[thick] (z1)--(y2) (z2)--(y3) (z3)--(y4) (z4)--(y5) (z5)--(y1);
\draw[thick, blue!50] (x2)--(x4)--(z4)--(y5)--(x5)--(x3)--(y3)--(z2)--(x2);
\end{scope}

\begin{scope}[xshift=12cm]
\foreach \x in {1,2,...,5} \draw[fill=black] (\x*72+18:1.8) {coordinate (x\x)} circle (0.1);
\draw[thick] (x1)--(x3)--(x5)--(x2)--(x4)--(x1);
\foreach \y in {1,2,3,4,5} 
{\draw[fill=black] (\y*72+6:3) {coordinate (y\y)} circle (0.1);
\draw[fill=black] (\y*72+30:3) {coordinate (z\y)} circle (0.1); 
\draw[thick] (x\y)--(y\y) (x\y)--(z\y);}
\draw[thick] (z1)--(y2) (z2)--(y3) (z3)--(y4) (z4)--(y5) (z5)--(y1);
\draw[thick, blue!50] (x2)--(x5)--(y5)--(z4)--(x4)--(x1)--(x3)--(y3)--(z2)--(x2);
\end{scope}

\end{tikzpicture}
\caption{The 6-, 7-, 8-, 9-cycles in a hypo-Petersen graph.}\label{FiCycle6789Hypo}
\end{figure}

    Let $C$ be a shortest cycle of $G$. By Claim \ref{ClNoC5}, $|C|\geqslant 6$. Note that any vertex outside $C$ has at most one neighbor in $C$ as otherwise it creates a cycle of length at most $|C|/2+2 < |C|$, contradicting the minimality of $C$. Since $G$ is 3-connected, every end block of a component of $G-V(C)$ has at least two neighbors in $C$. 
    
\begin{claim}\label{ClComponentAdmissiblePath}
    Let $Q$ be a component of $G-V(C)$ and $x_1,x_2\in N_C(Q)$, say $x_1y_1,x_2y_2\in E(G)$ with $y_1,y_2\in V(Q)$. 
    \begin{itemize}
        \item[(1)] If $Q$ is $2$-connected, then $G$ has $k-2$ admissible $(x_1,x_2)$-paths with all internal vertices in $Q$.
        \item[(2)] If $Q$ is not $2$-connected, $B$ is an end block of $Q$ and $y_1,y_2\in V(B)$, then $G$ has $k-2$ admissible $(x_1,x_2)$-paths with all internal vertices in $B$.
        \item[(3)] If $y_1,y_2$ are inner-vertices of two distinct end blocks of $Q$, then $G$ has $2k-5$ admissible $(x_1,x_2)$-paths with all internal vertices in $Q$.
    \end{itemize}
\end{claim}

\begin{proof}
    Note that every vertex in $Q$ has at most one neighbor in $C$. It follows that $y_1\neq y_2$ and every vertex in $Q$ has degree at least $k-1$ in $Q$. 

    (1) Set $G_1=G[V(Q)\cup\{x_1,x_2\}]$. Then $G_1+x_1x_2$ is 2-connected, and every vertex in $Q$ has degree at least $k-1$ in $G_1$. By Lemma \ref{thm: k adm paths}, $G_1$ contains $k-2$ admissible $(x_1,x_2)$-paths, as desired.

    (2) Let $c$ be the cut-vertex of $Q$ contained in $B$. Set $G_1=G[V(B)\cup\{x_1,x_2\}]$. Then $G_1+x_1x_2$ is 2-connected, and every vertex in $B-c$ has degree at least $k-1$ in $G_1$. By Lemma \ref{thm: k adm paths 2}, $G_1$ contains $k-2$ admissible $(x_1,x_2)$-paths, as desired.

    (3) Let $B_i$ be the end block of $Q$ containing $y_i$, let $c_i$ be the cut-vertex of $Q$ contained in $B_i$, and let $G_i=G[V(B_i)\cup\{x_i\}]$, $i=1,2$. So $G_i+x_ic_i$ is 2-connected, and every vertex in $B_i-c_i$ has degree at least $k-1$ in $G_i$. By Lemma \ref{thm: k adm paths}, $G_i$ contains $k-2$ admissible $(x_i,c_i)$-paths, $i=1,2$. Together with a $(c_1,c_2)$-path in $Q$, we can get $2k-5$ admissible $(x_1,x_2)$-paths of $G$ with all internal vertices in $Q$, as desired.
\end{proof}

\begin{claim}\label{claim: 2-connected components}
    Each component of $G-V(C)$ is $2$-connected. 
\end{claim}

\begin{proof}
    Let $Q$ be a component of $G-V(C)$ that is not 2-connected. If $k=5$, then by Claim \ref{ClComponentAdmissiblePath} (3), $G$ has $2k-5\geqslant 5$ admissible $(x_1,x_2)$-paths with all internal vertices in $Q$, where $x_1,x_2$ are two vertices in $C$ as in Claim \ref{ClComponentAdmissiblePath} (3). Together with an $(x_1,x_2)$-path in $C$, we get 5 admissible cycles in $G$, a contradiction. So we assume that $k=4$. 
    
    Let $B_1,B_2$ be two end blocks of $Q$, and let $c_1,c_2$ be the cut-vertices of $Q$ contained in $B_1,B_2$, respectively. We choose $B_1,B_2$ such that the distance of $c_1,c_2$ in $Q$ is as small as possible. 

    \begin{subclaim}\label{ClOneComponent}
        $G-V(C)$ has only one component $Q$, and if $Q$ has a third end block other than $B_1,B_2$, then $Q$ has only one cut-vertex $c_1(=c_2)$.
    \end{subclaim}

    \begin{proof}
        If $G-V(C)$ has a second component $Q'$, then let $B'=Q'$ if $Q'$ is 2-connected, and let $B'$ be an end block of $Q'$ if $Q'$ is not 2-connected. If $Q$ has a third end block other that $B_1,B_2$ and $c_1$ is not the only cut-vertex of $Q$, then we can choose an end block $B'$ of $Q$ that is disjoint with $B_1,B_2$. Let $P$ be a shortest path of $Q$ from $c_1$ to $c_2$. By the minimality of the distance between $c_1$ and $c_2$, the path $P$ does not pass through the cut-vertex of $Q$ contained in $B'$ (for the case $B'$ is contained in $Q$). It follows that $P$ and $B'$ are disjoint.

        By the 2-connectivity of $G$, let $x_1,x_2\in V(C)$ that has neighbors in $B_1-c_1$ and $B_2-c_2$, respectively, $x_1\neq x_2$; and let $x'_1,x'_2$ be two neighbors of $B'$ in $C$. By Claim \ref{ClComponentAdmissiblePath}, there are $2k-5=3$ admissible $(x_1,x_2)$-paths with all internal vertices in $B_1\cup B_2\cup P$, and there are $k-2=2$ admissible $(x'_1,x'_2)$-paths with all internal vertices in $B'$. Together with two disjoint paths in $C$ from $\{x_1,x_2\}$ to $\{x'_1,x'_2\}$, we can get four admissible cycles in $G$, a contradiction.
    \end{proof}

    \begin{subclaim}\label{ClTwoEndBlock}
        $Q$ has exactly two end blocks $B_1,B_2$.
    \end{subclaim}

    \begin{proof}
        Suppose otherwise. Let $B_1,B_2,\ldots,B_s$, $s\geqslant 3$, be the end blocks of $Q$. By Claim \ref{ClOneComponent}, $c_1=c_2:=c$ is the only cut-vertex of $G$. 
        
        Let $x_1,x_2\in V(C)$ be two quasi-diagonal vertices. We claim that $N_Q(x_1)\cup N_Q(x_2)\subseteq V(B_i)$ for some $i\in[1,s]$. Since $\delta(G)\geqslant 4$ and $Q$ is the only component of $G-V(C)$, every vertex in $C$ has at least two neighbors in $Q$, and at least one neighbor in $Q-c$. Suppose $x_1y_1,x_2y_2\in E(G)$ with $y_1\in V(B_i)\backslash\{c\}$ and $y_2\in V(B_j)\backslash\{c\}$, $i\neq j$. By Claim \ref{ClComponentAdmissiblePath} (3), $G$ has $2k-5=3$ admissible $(x_1,x_2)$-paths with all internal vertices in $Q$. Note that $C$ has two admissible $(x_1,x_2)$-paths. By Lemma \ref{lemma: concatenating paths}, $G$ has four admissible cycles, a contradiction. Thus, as we claimed, for each two quasi-diagonal vertices in $C$, their neighbors in $Q$ are contained in one end block of $Q$.

        Define $Qdi(C)$ the graph on $V(C)$ such that two vertices are adjacent in $Qdi(C)$ if and only if they are quasi-diagonal in $C$. If $C$ is odd or $|C|\equiv 0\bmod 4$, then $Qdi(C)$ is a cycle; and if $|C|\equiv 2\bmod 4$, then $Qdi(C)$ is the union of two disjoint cycles of length $|C|/2$. By the analysis above, we see that for each component $D$ of $Qdi(C)$, $N_Q(V(D))\subseteq V(B_i)$ for some end block $B_i$ of $Q$. It follows that $|C|\equiv 2\bmod 4$ and $Q$ has exactly two end blocks, a contradiction.
    \end{proof}

    \begin{subclaim}\label{Clz1z2CutEdge}
        Either $c_1=c_2$ or $c_1c_2\in E(Q)$ is a cut-edge of $Q$.
    \end{subclaim}

    \begin{proof}
        $G$ is 2-connected, so it contains two disjoint edges $x_1y_1,x_2y_2\in E(G)$ with $x_1,x_2\in V(C)$, $y_1\in V(B_1)\backslash\{c_1\}$ and $y_2\in V(B_2)\backslash\{c_2\}$. Suppose that $c_1\neq c_2$ and $c_1c_2$ is not a cut-edge of $Q$. Then $V(Q)\backslash(V(B_1)\cup V(B_2))\neq\emptyset$.       
        
        Let $G_1=G[V(B_1)\cup\{x_1\}]$, $G_2=G[V(B_2)\cup\{x_2\}]$, and $G_3=G[(V(Q)\backslash(V(B_1)\cup V(B_2)))\cup\{c_1,c_2\}]$. Note that $G_1+x_1c_1$ and $G_2+x_2c_2$ are 2-connected. By Claim \ref{ClTwoEndBlock}, $Q$ is a block-chain, implying that $G_3+c_1c_2$ is 2-connected. Since every vertex in $Q$ has degree at least 3, $G_1$ has two admissible $(x_1,c_1)$-paths, $G_3$ has two admissible $(c_1,c_2)$-paths, and $G_2$ has two admissible $(c_2,x_2)$-paths. It follows that $G$ has four admissible $(x_1,x_2)$-paths with all internal vertices in $Q$. Together with an $(x_1,x_2)$-path in $C$, we get four admissible cycles of $G$, a contradiction.
    \end{proof}

    We set $U_i=V(B_i)\cup\{c_1,c_2\}$, $i=1,2$. We claim that for each two quasi-diagonal vertices $x_1,x_2\in V(C)$, $N_Q(x_1)\cup N_Q(x_2)\subseteq U_i$ for some $i=1,2$. Recall that every vertex in $C$ has at least two neighbors in $Q$ and $G$ does not contain a triangle. It follows that every vertex in $C$ has at least one neighbor in $V(Q)\backslash\{c_1,c_2\}$. Suppose $x_1y_1,x_2y_2\in E(G)$ with $y_i\in V(B_i)\backslash\{c_i\}$, $i=1,2$. By Claim \ref{ClComponentAdmissiblePath} (3), $G$ has three admissible $(x_1,x_2)$-paths with all internal vertices in $Q$. Since $C$ has two admissible $(x_1,x_2)$-paths, by Lemma \ref{lemma: concatenating paths}, $G$ has four admissible cycles, a contradiction. Thus, as we claimed, $N_Q(x_1)\cup N_Q(x_2)\subseteq U_i$. It follows, as in the proof of Claim~\ref{ClTwoEndBlock}, that for each component $D$ of $Qdi(C)$, $N_Q(V(D))\subseteq U_i$ for some $i=1,2$. This implies that $Qdi(Q)$ has two components $D_1,D_2$ (and $|C|\equiv 2\bmod 4$), such that $N_Q(V(D_1))\subseteq U_1$ and $N_Q(V(D_2))\subseteq U_2$.

    Let $x_1,x_2$ be two sub-quasi-diagonal vertices in $C$, say $x_1\in V(D_1)$, $x_2\in V(D_2)$. It follows that $x_1$ has a neighbor $y_1\in V(B_1)\backslash\{c_1\}$ and $x_2$ has a neighbor $y_2\in V(B_2)\backslash\{c_2\}$. By Claim \ref{ClComponentAdmissiblePath} (3), $G$ has three admissible $(x_1,x_2)$-paths with all internal vertices in $Q$. Note that $C$ has two $(x_1,x_2)$-paths whose lengths differ by four. By Lemma \ref{lemma: concatenating paths a1a2k}, $G$ has four admissible cycles, a contradiction.
\end{proof}

\begin{claim}\label{ClCgeq7}
    $|C|\geqslant 7$.
\end{claim}

\begin{proof}
    Suppose that $|C|=6$. Let $Q$ be a component of $G-V(C)$. By Claim \ref{claim: 2-connected components}, $Q$ is 2-connected. First suppose that $N_C(Q)$ contains two adjacent vertices (that is, two sub-quasi-diagonal vertices). Let $x\in N_C(Q)$, $T=N_C(x)$ such that $N_T(Q)\neq\emptyset$. Let $G_1$ be the graph obtained from $G[V(Q)\cup\{x\}]$ by adding one vertex $x^*$ and all edges in $\{vx^*: v\in V(Q), N_T(v)\neq\emptyset\}$. Clearly $G_1+xx^*$ is 2-connected, and every vertex in $Q$ has degree at least $k-1$ in $G_1$. Moreover, for every edge $v_1v_2\in E(Q)$, either $v_1$ or $v_2$ has no neighbors in $V(C)\backslash(\{x\}\cup T)$; otherwise a cycle of length at most 5 appears. This implies that $v_1,v_2$ have degree sum at least $2k-1$ in $G_1$. By Theorem \ref{thm: adm paths degree sum}, $G_1$ has $k-1$ admissible $(x,x^*)$-paths of consecutive lengths or $k-2$ admissible $(x,x^*)$-paths whose lengths form an arithmetic progression with common difference two. It follows that $G$ has $k-1$ admissible $(x,T)$-paths of consecutive lengths or $k-2$ admissible $(x,T)$-paths whose lengths form an arithmetic progression with common difference two, with all internal vertices in $Q$. Notice that $C$ has two $(x,x')$-paths of lengths 1 and 5, respectively, for each $x'\in T$. By Lemma \ref{lemma: concatenating paths a1a2k}, $G$ has $k$ admissible cycles, a contradiction.
    
    Now suppose that $N_C(Q)$ contains no two adjacent vertices. Since $G$ is 3-connected, $|N_C(Q)|\geqslant 3$. It follows that $N_C(Q)$ consists of three vertices in $C$ each two of which have distance 2 (and thus quasi-diagonal in $C$). Set $N_C(Q)=\{x_1,x_2,x_3\}$. Let $G_2$ be the graph obtained from $G[V(Q)\cup\{x_1\}]$ by adding one vertex $x^*$ and all edges in $\{vx^*: v\in N_Q(x_2)\cup N_Q(x_3)\}$. It follows that $G_2+xx^*$ is 2-connected, and every vertex in $Q$ has degree at least $k$ in $G_2$. Thus, $G_2$ has $k-1$ admissible $(x,x^*)$-paths, implying that $G$ has $k-1$ admissible $(x,\{x_2,x_3\})$-paths. Note that for each $i\in\{2,3\}$, $C$ has two $(x_1,x_i)$-paths of lengths 2 and 4 respectively. By Lemma \ref{lemma: concatenating paths}, $G$ has $k$ admissible cycles, a contradiction.
\end{proof}

\begin{claim}\label{ClQxyAdmissiblePath}
    Let $Q$ be a component of $G-V(C)$ and $x,y\in N_C(Q)$. Then $G[V(Q)\cup\{x,y\}]$ has either $k-1$ admissible $(x,y)$-path of consecutive lengths, or $k-2$ admissible $(x,y)$-paths whose lengths form an arithmetic progression with common difference two.
\end{claim}

\begin{proof}
    By Claim \ref{ClCgeq7}, $|C|\geqslant 7$. Note that every vertex in $Q$ has degree at least $k-1$ in $Q$. For each edge $v_1v_2\in E(Q)$, either $N_C(v_1)=\emptyset$ or $N_C(v_2)=\emptyset$; otherwise we find a cycle shorter than $C$. It follows that every two adjacent vertices in $Q$ have degree sum at least $2k-1$ in $Q$. Now the assertion can be deduced by Theorem \ref{thm: adm paths degree sum}.
\end{proof}

\begin{claim}\label{ClSemiQuasiNQ}
    If $C$ is odd, then let $x_1,x_2\in V(C)$ be two quasi-diagonal vertices, while if $C$ is even, then let $x_1,x_2\in V(C)$ be two sub-quasi-diagonal vertices. For a component $Q$ of $G-V(C)$, either $N_Q(x_1)=\emptyset$ or $N_Q(x_2)=\emptyset$.
\end{claim}

\begin{proof}
    Suppose that $N_Q(x_1)\neq\emptyset$ and $N_Q(x_2)\neq\emptyset$. Let $G_1=G[V(Q)\cup\{x_1,x_2\}]$. By Claim \ref{ClQxyAdmissiblePath}, $G_1$ has either $k-1$ admissible $(x_1,x_2)$-path of consecutive lengths, or $k-2$ admissible $(x_1,x_2)$-path whose lengths form an arithmetic progression with common difference two. Note that $C$ contains two $(x_1,x_2)$-paths whose lengths differ by 1 in the odd case and 4 in the even case. By Lemma \ref{lemma: concatenating paths a1a2k}, $G$ has $k$ admissible cycles, a contradiction.   
\end{proof}

\begin{claim}\label{ClExactlyTwoComponent}
    $G-V(C)$ has exactly two components.
\end{claim}

\begin{proof}
    If $G-V(C)$ has only one component, then every vertex in $C$ has a neighbor this component, contradicting Claim \ref{ClSemiQuasiNQ}.

    Now suppose that $G-V(C)$ has at least three components, say $Q_1,Q_2,Q_3$. Let $G_0$ be the graph obtained from $G$ by contracting all edges in $Q_1$, $Q_2$ and $Q_3$. By Lemma \ref{Le3ConnectedContract}, $G_0$ is 3-connected. Let $w_i$ be the vertex in $G_0$ obtained by contracting $Q_i$, $i=1,2,3$. It follows that $G_0$ has a cycle $C_0$ containing $w_1,w_2,w_3$. Let $x_i$ and $y_i$ be the predecessor and successor of $w_i$ in $C_0$, $i=1,2,3$. So $x_i,y_i\in N_C(Q_i)$. Let $G_i=G[V(Q_i)\cup\{x_i,y_i\}]$, $i=1,2,3$. By Claim \ref{ClComponentAdmissiblePath}, $G_i$ has $k-2$ admissible $(x_i,y_i)$-paths. Together with the three paths in $C_0-\{w_1,w_2,w_3\}$, we can get $3k-8\geqslant k$ admissible cycles of $G$, a contradiction. 
\end{proof}

By Claim \ref{ClExactlyTwoComponent}, let $Q_1,Q_2$ be the two components of $G-V(C)$. We fix an orientation of $C$. 

\begin{claim}\label{Clx1x2y1y2}
    $C$ is even and there exist $y_1,y_2,z_1,z_2\in V(C)$ in this order along $C$ such that $y_1,z_1\in N_C(Q_1)$, $y_2,z_2\in N_C(Q_2)$ and $|\overrightarrow{C}[y_1,y_2]|+|\overrightarrow{C}[z_1,z_2]|=|C|/2-1$.
\end{claim}

\begin{proof}
    By Claim \ref{ClSemiQuasiNQ}, every vertex in $C$ has neighbors in exactly one of $Q_1$, $Q_2$. Set $C=x_1x_2\ldots x_{|C|}x_1$. For convenience, we color all the vertices of $N_C(Q_1)$ red and all the vertices in $N_C(Q_2)$ blue. 
    
    If $C$ is odd, then by Claim~\ref{ClSemiQuasiNQ}, any two quasi-diagonal vertices have distinct colors. Since for any $i$, vertices $x_i$ and $x_{i+1}$ (the indices are taken modulo $|C|$) have a common quasi-diagonal vertex, they have the same color. Thus, all vertices have the same color, contradicting Claim~\ref{ClSemiQuasiNQ}. 
    
    Similarly, if $C$ is even, then by Claim~\ref{ClSemiQuasiNQ}, for any $i$, vertices $x_i$ and $x_{i+4}$ are of the same color, because they have a common sub-quasi-diagonal vertex. 

    If $|C|\equiv 2\bmod 4$, then all vertices $x_i$ with $i$ odd have the same color, say red, and all other vertices in $C$ have the same color blue. We can choose $y_1=x_1$, $y_2=x_2$, $z_1=x_{|C|/2}$, $z_2=x_{|C|-2}$. One can compute that $y_1,y_2,z_1,z_2$ meet the conclusion of the assertion. So we assume that $|C|\equiv 0\bmod 4$.

    If $|C|\equiv 4\bmod 8$, then $y_1$ and $y_{|C|/2+3}$ have the same color. However, $y_1$ and $y_{|C|/2+3}$ are sub-quasi-diagonal in $C$, a contradiction.

    Now assume that $|C|\equiv 0\bmod 8$. Set $X_i=\{x_j: j\equiv i\bmod 4\}$, $i=1,2,3,4$. So all vertices in $X_i$ have the same color. Since $x_1\in X_1$ and $x_{|C|/2+3}\in X_3$ are sub-quasi-diagonal, the vertices of $X_1$ and of $X_3$ have distinct colors, and similarly, the vertices of $X_2$ and of $X_4$ have distinct colors. We can assume w.l.o.g. that the vertices of $X_1\cup X_2$ are red and of $X_3\cup X_4$ are blue. Let $y_1=x_2$, $y_2=x_3$, $z_1=x_{|C|/2+2}$, $z_2=x_{|C|}$. One can compute that $y_1,y_2,z_1,z_2$ meet the conclusion of the assertion. 
\end{proof}

Let $y_1,y_2,z_1,z_2\in V(C)$ be vertices as in Claim \ref{Clx1x2y1y2}. Set $G_1=G[V(Q_1)\cup\{y_1,z_1\}]$ and $G_2=G[V(Q_2)\cup\{y_2,z_2\}]$. Then $G_i$ contains $k-2$ admissible $(y_i,z_i)$-paths for $i=1,2$. For a $(y_1,z_1)$-path $P_1$ in $G_1$ and a $(y_2,z_2)$-path $P_2$ in $G_2$, we can get two cycles $P_1\cup P_2\cup\overrightarrow{C}[y_1,y_2]\cup\overrightarrow{C}[z_1,z_2]$ and $P_1\cup P_2\cup\overrightarrow{C}[y_2,z_1]\cup\overrightarrow{C}[z_2,y_1]$, which are of lengths $|P_1|+|P_2|+|C|/2-1$ and $|P_1|+|P_2|+|C|/2+1$, respectively. It follows that we can get $k$ admissible cycles of $G$, a contradiction. 

This finishes the proof of Lemma~\ref{lemma: k adm cycles for k=4,5}.
\end{proof}

\begin{proof}[\bf Proof of Theorem \ref{thm: k adm cycles min degree k}]
    By Lemma \ref{lemma: high connectivity no k adm cycles}, we can assume that the connectivity of $G$ is at least $(k+2)/2\geqslant 3$. If $G$ has a triangle or a 4-cycle, then we are done by Theorem \ref{thm: 2-con-K3-k-cycles} or Lemma \ref{lemma: 3-con-grith4}. If $k\in\{4,5\}$, then we are done by Lemma \ref{lemma: k adm cycles for k=4,5}; and if $k\geqslant 6$, then we are done by Theorems \ref{thm: 3-con-bi} and \ref{thm: 3-con-noK3-k-cycles}.
\end{proof}

\section{Cycles of lengths $\ell$ modulo $k$ for all even $\ell$}\label{section: even ell}

In this section, we prove Theorems \ref{thm: l mod k min degree k} and \ref{thm: l mod k min degree k-1}. Corollary \ref{corollary: l mod k} can be deduced from Theorem \ref{thm: l mod k min degree k}. The following results will be used in our proofs.

\begin{theorem}[Bondy and Vince \cite{BV98}]\label{thm: 2 adm cycles}
    With the exception of $K_1$ and $K_2$, every connected graph having at most two vertices of degree less than $3$ contains two admissible cycles.
\end{theorem}

\begin{theorem}[Gao, Li, Ma and Xie \cite{GLMX}]\label{thm: 3-con-2-consecutive-even-cycles}
    Every $3$-connected graph with order at least $6$ contains two cycles of consecutive even lengths.
\end{theorem}

\begin{theorem}[Dean, Lesniak and Saito \cite{DLS93}]\label{thm: 0 mod 4 min deg 2}
    Every graph with minimum degree at least $2$ and with at most two vertices of degree $2$ contains a cycle of length $0$ modulo $4$. 
\end{theorem}

\begin{lemma}\label{lemma: two paths differing by 2 mod 4}
    Let $G$ be a graph and $x,y$ two vertices of $G$. If $G+xy$ is $2$-connected and every vertex of $G$ other than $x,y$ has degree at least $3$, then at least one of the following holds:
    \begin{itemize}
        \item[(1)] $G$ contains a cycle of length $2$ modulo $4$;
        \item[(2)] $G$ contains two $(x,y)$-paths whose lengths differ by $2$ modulo $4$; or
        \item[(3)] $G$ has the construction $L_i(x,y)$ for some $i=1,\ldots,4$ (see Figure \ref{Fig: 4 possibilities of G}).
    \end{itemize}
\end{lemma}

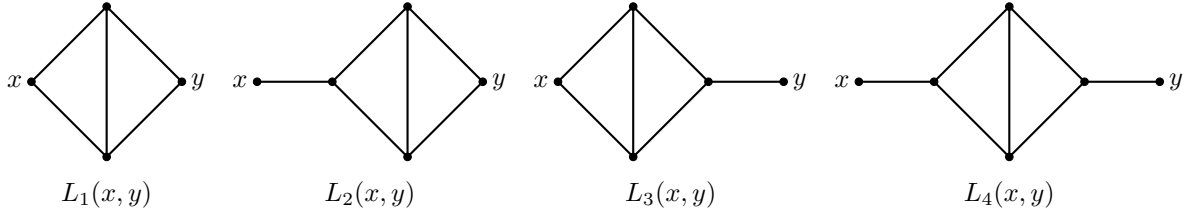
\begin{figure}[ht]
\centering
\begin{tikzpicture}[scale=0.5]

\begin{scope}[xshift=-12cm]
\draw[fill=black] (-2,0) {coordinate (x)} {node[left] {$x$}} circle (0.1);
\draw[fill=black] (2,0) {coordinate (y)} {node[right] {$y$}} circle (0.1);
\draw[fill=black] (0,-2) {coordinate (v)} circle (0.1);
\draw[fill=black] (0,2) {coordinate (u)} circle (0.1);
\draw[thick] (x)--(u)--(y) (x)--(v)--(y) (u)--(v);
\node at (0,-3) {$L_1(x,y)$};
\end{scope}

\begin{scope}[xshift=-4cm]
\draw[fill=black] (-4,0) {coordinate (x)} {node[left] {$x$}} circle (0.1);
\draw[fill=black] (-2,0) {coordinate (x1)} circle (0.1);
\draw[fill=black] (2,0) {coordinate (y)} {node[right] {$y$}} circle (0.1);
\draw[fill=black] (0,-2) {coordinate (v)} circle (0.1);
\draw[fill=black] (0,2) {coordinate (u)} circle (0.1);
\draw[thick] (x)--(x1) (x1)--(u)--(y) (x1)--(v)--(y) (u)--(v);
\node at (-1,-3) {$L_2(x,y)$};
\end{scope}

\begin{scope}[xshift=2cm]
\draw[fill=black] (-2,0) {coordinate (x)} {node[left] {$x$}} circle (0.1);
\draw[fill=black] (2,0) {coordinate (y1)} circle (0.1);
\draw[fill=black] (4,0) {coordinate (y)} {node[right] {$y$}} circle (0.1);
\draw[fill=black] (0,-2) {coordinate (v)} circle (0.1);
\draw[fill=black] (0,2) {coordinate (u)} circle (0.1);
\draw[thick] (x)--(u)--(y1) (x)--(v)--(y1) (y1)--(y) (u)--(v);
\node at (1,-3) {$L_3(x,y)$};
\end{scope}

\begin{scope}[xshift=12cm]
\draw[fill=black] (-4,0) {coordinate (x)} {node[left] {$x$}} circle (0.1);
\draw[fill=black] (-2,0) {coordinate (x1)} circle (0.1);
\draw[fill=black] (2,0) {coordinate (y1)} circle (0.1);
\draw[fill=black] (4,0) {coordinate (y)} {node[right] {$y$}} circle (0.1);
\draw[fill=black] (0,-2) {coordinate (v)} circle (0.1);
\draw[fill=black] (0,2) {coordinate (u)} circle (0.1);
\draw[thick] (x)--(x1) (x1)--(u)--(y1) (x1)--(v)--(y1) (y1)--(y) (u)--(v);
\node at (0,-3) {$L_4(x,y)$};
\end{scope}

\end{tikzpicture}
\caption{Constructions of $L_i(x,y)$ ($i=1,\ldots,4$) in Lemma \ref{lemma: two paths differing by 2 mod 4}.}\label{Fig: 4 possibilities of G}
\end{figure}

\begin{proof}
Suppose that $G$ is a counterexample with the smallest number of edges. Since $G+xy$ is 2-connected, we have that $G$ is 2-connected or a block-chain.

\setcounter{claim}{0}
\begin{claim}\label{claim: xynotinE}
    $xy\notin E(G)$.    
\end{claim}

\begin{proof}
    Suppose that $xy\in E(G)$. Let $G_0=G-xy$. Then $G_0$ satisfies the condition of the theorem. Note that $|E(G_0)|<|E(G)|$. It follows that $G_0$ satisfies (1), (2) or (3). If $G_0$ contains a cycle of length 2 modulo 4, or $G_0$ contains two $(x,y)$-paths whose lengths differ by 2 modulo 4, then so does $G$, a contradiction. Now assume that $G_0$ has the construction $L_i(x,y)$ for some $i=1,\ldots,4$. If $i=1, 2$ or 3, then $G_0$ has an $(x,y)$-path of length 3, together with $xy$, we get two $(x,y)$-paths in $G$ of lengths differing by 2 modulo 4, a contradiction. So $i=4$. Then $G_0$ has an $(x,y)$-path of length 5, together with $xy$, we get a cycle of length 6 in $G$, a contradiction.
\end{proof}

\begin{claim}\label{claim: no two odd paths}
    If $C$ is an even cycle of $G$ and $P^x,P^y$ two disjoint paths from $C$ to $\{x,y\}$, say with origins $x',y'$, respectively, then $d_C(x',y')$ is even. 
\end{claim}

\begin{proof}
    Since $G$ has no cycles of length 2 modulo 4, we have $|C|\equiv 0\bmod 4$.
    Fix an orientation of $C$.
    If $d_C(x',y')$ is odd, then both $\overrightarrow{C}[x',y']$ and $\overleftarrow{C}[x',y']$ are odd. We can assume w.l.o.g. that $|\overrightarrow{C}[x',y']|\equiv 1\bmod~4$ and $|\overleftarrow{C}[x',y']|\equiv 3\bmod~4$. Now $P^xx'\overrightarrow{C}[x',y']y'P^y$ and $P^xx'\overleftarrow{C}[x',y']y'P^y$ are two $(x,y)$-paths of lengths differing by 2 modulo 4, a contradiction.
\end{proof}

We claim that $G$ contains an even cycle. If not, then $G$ is 2-connected or a block-chain such that each block is an odd cycle or an edge. This implies that $G$ contains a vertex, different from $x,y$, with degree 2, a contradiction. Let $G^*$ be the graph obtained from $G$ by adding a new vertex $w$ and two edges $xw,yw$. By $G+xy$ being 2-connected, $G^*$ is 2-connected. Let $C$ be an even cycle of $G$ such that the component $H$ of $G^*-V(C)$ containing $w$ is as large as possible and subject to this, $|C|$ is as small as possible.

Now let $P^x, P^y$ be two disjoint paths from $C$ to $\{x,y\}$, say with origins $x',y'$, respectively. We remark that possibly $x=x'$ or $y=y'$ or both.

\begin{claim}\label{claim: evendistance}
    Every two vertices in $N_C(H)$ have an even distance in $C$.
\end{claim}

\begin{proof}
    By Claim \ref{claim: no two odd paths}, $d_C(x',y')$ is even. In the following, we will show that every vertex in $N_C(H)$ has an even distance to $x'$ or to $y'$ in $C$, following which the assertion can be deduced directly.

    Let $z\in N_C(H)$. Then there is a path $P$ from $z$ to $P^x\cup P^y-\{x',y'\}$ with all internal vertices in $H$. Let $z'$ be the end-vertex of $P$ and assume w.l.o.g. that $z'\in V(P^x)\backslash\{x'\}$. Now $Pz'P^x[z',x]$ and $P^y$ are two disjoint paths from $C$ to $\{x,y\}$. By Claim \ref{claim: no two odd paths}, $d_C(z,y')$ is even.
\end{proof}

\begin{claim}\label{claim: mostonechord}
    $C$ has at most one chord.
\end{claim}

\begin{proof}
    Suppose otherwise that $C$ has at least two chords. If $|C|=4$, then $G[V(C)]\cong K_4$, implying that there are two $(x',y')$-paths in $G[V(C)]$ of lengths 1 and 3, respectively. Together with $P^x,P^y$, we get two $(x,y)$-paths of lengths differing by 2 modulo 4, a contradiction. Now assume that $|C|\geqslant 8$. Then there is a $\varTheta$-graph $T$ with $V(T)\subset V(C)$, which contains an even cycle $C_1$. Hence $|C_1|<|C|$ and the component of $G^*-V(C_1)$ containing $w$ will also contain $H$, contradicting the choice of $C$. 
\end{proof}

\begin{claim}\label{claim: onecomponentH}
    $G^*-V(C)$ has only one component.
\end{claim}

\begin{proof}
    Let $Q$ be a component of $G^*-V(C)$ other than $H$.

    \begin{subclaim}\label{claim: NCQgeq3}
        $|N_C(Q)|\geqslant 3$.
    \end{subclaim}

    \begin{proof}
        Assuming otherwise, by the 2-connectivity of $G$, we have $|N_C(Q)|=2$, say $N_C(Q)=\{x_1,y_1\}$. Let $G_1=G[V(Q)\cup\{x_1,y_1\}]$. Then $G_1+x_1y_1$ is 2-connected and every vertex of $G_1$ other than $x_1,y_1$ has the same degree as in $G$. It follows that $G_1$ satisfies the conclusion (1), (2) or (3). If $G_1$ contains a cycle of length 2 modulo 4, then so does $G$, a contradiction. If $G_1$ contains two $(x_1,y_1)$-paths of lengths differing by 2 modulo 4, then together with two disjoint paths from $\{x_1,y_1\}$ to $\{x,y\}$, we get two $(x,y)$-path in $G$ of lengths differing by 2 modulo 4, a contradiction. Now we conclude that $G_1$ has construction $L_i(x_1,y_1)$ for some $i=1,\ldots,4$.

        Let $C_1$ be the 4-cycle of $G_1$. If $\{x_1,y_1\}\neq\{x',y'\}$, or if $G_1$ has construction $L_i(x_1,y_1)$ with $i\in\{2,3,4\}$, then the component of $G^*-V(C_1)$ containing $w$ will contain $H$ and at least one vertex in $\{x',y'\}$, contradicting the choice of $C$. So we conclude that $\{x_1,y_1\}=\{x',y'\}$ and $G_1$ has construction $L_1(x_1,y_1)$.

        Let $G_2=G-Q$. If $x'\neq x$, then $x'$ has two neighbors in $C$ and one neighbor in $P^x$. If $y'\neq y$, then $y'$ has two neighbors in $C$ and one neighbor in $P^y$. In any case we see that every vertex of $G_2$ other than $x,y$ has degree at least 3 in $G_2$. It follows that $G_2$ satisfies the conclusion (1), (2) or (3). If $G_2$ contains a cycle of length 2 modulo 4, or $G_2$ contains two $(x,y)$-paths of lengths differing by 2 modulo 4, then so does $G$, a contradiction. Now we conclude that $G_2$ has construction $L_i(x,y)$ for some $i=1,\ldots,4$.

        Note that $C$ is an even cycle in $G_2$, so $|C|=4$ and $C$ has a chord $zz'$, where $\{z,z'\}=V(C)\backslash\{x',y'\}$. Also note that $G_1$ has construction $L_1(x',y')$, i.e., $G_1$ is a 4-cycle with a chord $z_1z'_1$, where $\{z_1,z'_1\}=V(G_1)\backslash\{x',y'\}$. Now $x'zz'y'z_1z'_1x'$ is a 6-cycle of $G$, a contradiction. 
    \end{proof}

    \begin{subclaim}\label{claim: NCQeq3}
        $|N_C(Q)|=3$ and $N_C(H)=\{x',y'\}\subset N_C(Q)$.
    \end{subclaim}

    \begin{proof}
        If $|N_C(Q)|\geqslant 4$, or $N_C(H)\backslash N_C(Q)\neq\emptyset$, then there are three vertices $u_1,u_2,u_3\in N_C(Q)$ and one vertex $z\in N_C(H)\backslash\{u_1,u_2,u_3\}$. Assume w.l.o.g. that $u_1,u_2,u_3,z$ appear in this order along $C$. It follows that there are three internally-disjoint paths $P_1,P_2,P_3$ from some $u\in V(Q)$ to $u_1,u_2,u_3$, respectively, with all internal vertices in $Q$. Now $P_1\cup P_2\cup P_3\cup\overrightarrow{C}[u_1,u_3]$ is a $\varTheta$-graph, which contains an even cycle $C_1$. The component of $G^*-C_1$ containing $w$ will contain $H$ and $z$, a contradiction. So we conclude that $|N_C(Q)|=3$ and $N_C(H)\subseteq N_C(Q)$.

        If $|N_C(H)|=3$, say $N_C(H)=\{x',y',z\}$, then $N_C(Q)=\{x',y',z\}$ as well. We can assume w.l.o.g. that $x',y',z$ appear in this order along $C$. By Claim \ref{claim: evendistance}, all three paths $\overrightarrow{C}[x',y']$, $\overrightarrow{C}[y',z]$, $\overrightarrow{C}[z,x']$ are even. Now let $P_1,P_2,P_3$ be three internally-disjoint paths from some $u\in V(Q)$ to $x',y',z$, respectively, with all internal vertices in $Q$. It follows that one of the three paths $P_1uP_2$, $P_2uP_3$, $P_3uP_1$ is even, implying that one of the three cycles $P_2uP_1x'\overrightarrow{C}[x',y']$, $P_3uP_2y'\overrightarrow{C}[y',z]$, $P_1uP_3z\overrightarrow{C}[z,x']$ is even. Let $C_1$ be such an even cycle. The component of $G^*-C_1$ containing $w$ will contain $H$ and one vertex in $\{x',y',z\}$, a contradiction. Hence we conclude that $|N_C(H)|=2$, i.e., $N_C(H)=\{x',y'\}$.
    \end{proof}

    By Claim \ref{claim: NCQeq3}, we can assume that $N_C(Q)=\{x',y',z\}$. We assume w.l.o.g. that $z\in V(\overrightarrow{C}[x',y'])$.

    \begin{subclaim}\label{claim: onecomponentQ}
        $Q$ is the only component of $G^*-V(C)$ other than $H$.
    \end{subclaim}

    \begin{proof}
        Let $Q'$ be a component of $G^*-V(C)$ other than $H$ and $Q$. By the analysis in Claims \ref{claim: NCQgeq3} and \ref{claim: NCQeq3}, $N_C(Q')=\{x',y',z'\}$ for some $z'\in V(C)\backslash\{x',y'\}$. Let $P$ be an $(x',z)$-path with all internal vertices in $Q$; and let $P'$ be an $(x',z')$-path with all internal vertices in $Q'$.

        First we assume that $z'\in V(\overrightarrow{C}[x',y'])$. We can assume w.l.o.g. that $x',z,z',y'$ appear in this order along $C$. Now $\overrightarrow{C}[x',z']\cup P\cup P'$ is a $\varTheta$-graph, which contains an even cycle $C_1$. The component of $G^*-C_1$ containing $w$ will contain $H$ and $y'$, a contradiction.

        Now we assume that $z'\in V(\overleftarrow{C}[x',y'])$. If the cycle $\overrightarrow{C}[x',z]zPx'$ is of even length, then the component of $G^*-(V(\overrightarrow{C}[x',z])\cup V(P))$ containing $w$ will contain $H$ and $y'$, contradicting the choice of $C$. So we conclude that $\overrightarrow{C}[x',z]zPx'$ is odd, implying that $Pz\overrightarrow{C}[z,y']$ is an odd $(x',y')$-path. By a similar analysis we see that $P'z'\overleftarrow{C}[z',y']$ is an odd $(x',y')$-path, contradicting Claim \ref{claim: no two odd paths} for the cycle $x'Pz\overrightarrow{C}[z,y']\overrightarrow{C}[y',z']z'P'x'$.
    \end{proof}

    \begin{subclaim}\label{claim: Ceq4}
        $|C|=4$ and $C$ has exactly one chord $zz'$, where $\{z'\}=V(C)\backslash\{x',y',z\}$.
    \end{subclaim}

    \begin{proof}
        By Claim \ref{claim: mostonechord}, $C$ has at most one chord. By Claims \ref{claim: NCQeq3} and \ref{claim: onecomponentQ}, every vertex in $V(C)\backslash\{x',y',z\}$ has no neighbors outside $C$. If $|C|\geqslant 8$, then some vertices in $V(C)\backslash\{x',y',z\}$ are not incident to a chord of $C$ and have degree 2, a contradiction. So we conclude that $|C|=4$.
        
        Set $C=x'zy'z'x'$. Notice that $z'\notin N_C(H)\cup N_C(Q)$. We have that $z'$ is incident to the chord of $C$, i.e., $zz'\in E(G)$.
    \end{proof}

    Let $P$ be a path from $x'$ to $z$ with all internal vertices in $Q$. Then $P\cup \{x'z,x'z',zz'\}$ is a $\varTheta$-graph, which contains an even cycle $C'$ not passing through $y'$. Consequently, the component of $G - C'$ containing $w$ is larger than $H$, which gives a contradiction.
\end{proof}

\begin{claim}\label{claim: Ceq4-noQ}
    $|C|=4$ and $C$ has exactly one chord $zz'$, where $\{z,z'\}=V(C)\backslash\{x',y'\}$.
\end{claim}

\begin{proof}
    By Claim \ref{claim: mostonechord}, $C$ has at most one chord. By Claims \ref{claim: evendistance} and \ref{claim: onecomponentH}, there are at least $|C|/2$ vertices in $C$ that have no neighbors outside $C$. If $|C|\geqslant 8$, then some vertices in $V(C)\backslash N_C(H)$ are not incident to a chord of $C$ and have degree 2, a contradiction. So we conclude that $|C|=4$.
        
    Set $C=x'zy'z'x'$. Notice that $z,z'\notin N_C(H)$. We have that $zz'$ is a chord of $C$.
\end{proof}

\begin{claim}\label{claim: nopathfromPxtoPy}
    $G-\{z,z'\}$ contains no path from $P^x$ to $P^y$.
\end{claim}

\begin{proof}
    Suppose otherwise that $G-\{z,z'\}$ contains a path from $P^x$ to $P^y$.

    \begin{subclaim}\label{claim: pathx-y-even}
        Every path of $G-\{z,z'\}$ from $P^x$ to $P^y$ is even.
    \end{subclaim}

    \begin{proof}
        Suppose that $P$ is an odd path of $G-\{z,z'\}$ from $P^x$ to $P^y$. Let $x_1$, $y_1$ be the end-vertices of $P$. Notice that either $P^x[x_1,x']x'zy'P^y[y',y_1]$ or $P^x[x_1,x']x'zz'y'P^y[y',y_1]$ is an odd $(x_1,y_1)$-path, together with $P$ we get an even cycle, say $C_1$. Now $P^x[x_1,x]$ and $P^y[y_1,y]$ are two disjoint paths from $C_1$ to $\{x,y\}$ such that $d_{C_1}(x_1,y_1)$ is odd, contradicting Claim \ref{claim: no two odd paths}. 
    \end{proof}
    
    We may assume that $P^x$ and $P^y$ are disjoint paths from $x$ to $x'$ and $y$ to $y'$, respectively, such that the distance between $P^x$ and $P^y$ in $G-\{z,z'\}$ is the shortest. Let $P$ be the shortest path of $G-\{z,z'\}$ from $P^x$ to $P^y$, say from $x_1$ to $y_1$. By Claim \ref{claim: pathx-y-even}, $P$ is even, implying that $V(P)\backslash\{x_1,y_1\}\neq\emptyset$. Let $Q$ be the component of $G-(\{z,z'\}\cup V(P^x)\cup V(P^y))$ containing $P-\{x_1,y_1\}$. 
    
    \begin{subclaim}\label{claim: NPxPyQ2}
        $N_{P^x\cup P^y}(Q)=\{x_1,y_1\}$.
    \end{subclaim}

    \begin{proof}
        Suppose that $|N_{P^x\cup P^y}(Q)|\geqslant 3$, say $u\in N_{P^x\cup P^y}(Q)\backslash\{x_1,y_1\}$. We can assume w.l.o.g. that $u\in V(P^x)$. There is a path $P'$ from a vertex $u' \in V(P) \backslash \{x_1,y_1\}$ to $u$ with all internal vertices in $Q$. Now let $P_1^x$ be the path obtained from $P^x$ by using $P[x_1,u']u'P'$ instead of $P^x[x_1,u]$. Then $P_1^x$ and $P^y$ are two disjoint paths of $G$ from $C$ to $\{x,y\}$, and $P[u',y_1]$ is a path from $P_1^x$ to $P^y$ in $G-\{z,z'\}$. Clearly $|P[u',y_1]|<|P|$, contradicting the choice of $P^x$ and $P^y$.
    \end{proof}
    
    Let $G_4=G[V(Q)\cup\{x_1,y_1\}]$. By Claim \ref{claim: NPxPyQ2}, $G_4+x_1y_1$ is 2-connected and every vertex of $G_4$ other than $x_1,y_1$ has the same degree as that in $G$. It follows that $G_4$ satisfies the conclusion (1), (2) or (3). If $G_4$ contains a cycle of length 2 modulo 4, then so does $G$, a contradiction. If $G_4$ contains two $(x_1,y_1)$-paths of lengths differing by 2 modulo 4, then together with the two paths $P^x[x_1,x]$ and $P^y[y_1,y]$, we get two $(x,y)$-path in $G$ of lengths differing by 2 modulo 4, a contradiction. Now we conclude that $G_4$ has construction $L_i(x_1,y_1)$ for some $i=1,\ldots,4$. It follows that $G_4$ contains an odd path from $P^x$ to $P^y$, contradicting Claim \ref{claim: pathx-y-even}.
\end{proof}

\begin{claim}\label{claim: x=x'orNx=x'}
    Either $x=x'$ or $N(x)=\{x'\}$. Similarly, either $y=y'$ or $N(y)=\{y'\}$.
\end{claim}

\begin{proof}
    Suppose that $x\neq x'$. By Claim \ref{claim: nopathfromPxtoPy}, $x'$ is a cut-vertex of $G$ separating $x$ with $C-x'$. Let $H'$ be the component of $G-x'$ containing $x$ and let $G_5=G[V(H')\cup\{x'\}]$. If $N(x)\neq\{x'\}$, then $|V(G_5)|\geqslant 3$, and every vertex of $G_5$ other than $x,x'$ has degree at least 3 in $G_5$. It follows that $G_5$ satisfies the conclusion (1), (2) or (3). If $G_5$ contains a cycle of length 2 modulo 4, then so does $G$, a contradiction. If $G_5$ contains two $(x,x')$-paths of lengths differing by 2 modulo 4, then together with $\overrightarrow{C}[x',y']y'P^yy$, we get two $(x,y)$-paths in $G$ of lengths differing by 2 modulo 4, a contradiction. Now we conclude that $G_5$ has construction $L_i(x,x')$ for some $i=1,\ldots,4$. It follows that $G_5$ contains two $(x,x')$-paths of lengths differing by 1. Notice that $G[V(C)]$ contains two $(x',y')$-paths of lengths differing by 1. It follows that $G$ contains two $(x,y)$-paths of lengths differing by 2, a contradiction.

    The second assertion can be proved similarly.
\end{proof}

By Claims \ref{claim: onecomponentH}, \ref{claim: Ceq4-noQ} and \ref{claim: x=x'orNx=x'}, we see that $G$ has construction $L_i(x,y)$ for some $i=1,\ldots,4$, a contradiction.
\end{proof}

\begin{lemma}\label{lemma: 2 mod 4 cycle}
    Let $G$ be a $2$-connected graph on at least $6$ vertices with minimum degree at least $3$. Then $G$ contains a cycle of length $2$ modulo $4$.
\end{lemma}

\begin{proof}
    Assume that $G$ does not contain a $(2\bmod 4)$-cycle. By Theorem \ref{thm: 3-con-2-consecutive-even-cycles}, the connectivity of $G$ is 2, say $\{x,y\}$ is a vertex cut of $G$. Let $Q_1,Q_2$ be two of the components of $G-\{x,y\}$. We can choose the cut $\{x,y\}$ such that the component $Q_1$ is as small as possible. Denote by $G_i=G[V(Q_i)\cup \{x,y\}]$, where $i=1,2$. By $G$ being 2-connected, $G_i+xy$ is 2-connected, and every vertex of $G_i$ other than $x,y$ has degree at least 3 in $G_i$, $i=1,2$. Notice that $G_i$ contains no cycle of length 2 modulo 4. 
    
    We suppose first that $G_1$ contains two $(x,y)$-paths of lengths differing by 2 modulo 4, say $P_1, P_2$. If $G_2$ contains an odd cycle $C$, then $G_2$ contains two $(x,y)$-paths of different parities. Let $P$ be an $(x,y)$-path of $G_2$ such that $P,P_1$ have lengths of the same parity. Now $P\cup P_1$ and $P\cup P_2$ are two even cycles of lengths differing by 2 modulo 4, one of which is of length 2 modulo 4, a contradiction. Now we assume that $G_2$ contains no odd cycles, implying that $G_2$ is bipartite. By Theorem \ref{thm: 2 adm cycles}, $G_2$ contains two admissible cycles, which have consecutive even lengths, and one of which is a $(2\bmod 4)$-cycle, a contradiction. So we conclude that $G_1$ contains no two $(x,y)$-paths of lengths differing by 2 modulo 4.

    By Lemma \ref{lemma: two paths differing by 2 mod 4}, $G_1$ has construction $L_i(x,y)$ for some $i=1,\ldots,4$. By a similar analysis we can prove that $G_2$ has construction $L_j(x,y)$ for some $j=1,\ldots,4$. Recall that we choose $\{x,y\}$ such that $Q_1$ is as small as possible. So $G_1$ has construction $L_1(x,y)$ and it contains two $(x,y)$-paths of lengths 2 and 3, respectively. Also note that $G_2$ has an $(x,y)$-path of length either 3 or 4. It follows that $G$ contains a cycle of length 6, a contradiction.
\end{proof}

The following result will be used in our proof.

\begin{theorem}\label{thm: ChenSaitoDeanetal}
    Let $G$ be a $2$-connected graph with minimum degree at least $3$. Then
    \begin{itemize}
        \item[(1)] (Chen and Saito \cite{CS94}) $G$ has a cycle of length $0$ modulo $3$;
        \item[(2)] (Dean et al. \cite{D+91}) $G$ has a cycle of length $1$ modulo $3$ unless $G$ is a Petersen graph;
        \item[(3)] (Dean et al. \cite{D+91}) $G$ has a cycle of length $2$ modulo $3$ unless $G$ is isomorphic to $K_4$ or $K_{3,n-3}$.
    \end{itemize}
\end{theorem}

Now we are ready to prove Theorems \ref{thm: l mod k min degree k} and \ref{thm: l mod k min degree k-1}.

\begin{proof}[\bf Proof of Theorem \ref{thm: l mod k min degree k}] 
    If $k$ is odd, then the existence of $k$ admissible cycles implies the existence of $(\ell\bmod k)$-cycles for all even $\ell$. The assertion follows by Theorem \ref{thm: k adm cycles min degree k}. So assume that $k$ is even. By Theorem \ref{thm: 0 mod 4 min deg 2} and Lemma \ref{lemma: 2 mod 4 cycle}, the assertion holds for $k=4$. So assume that $k\geqslant 6$. 
    
    Assume that $G$ does not contain an $(\ell\bmod k)$-cycle for some even $\ell$ and $G$ is isomorphic to neither $K_{k+1}$ nor $K_{k,n-k}$. By Lemma \ref{lemma: high connectivity no l mod k cycle}, the connectivity of $G$ is at least $(k+2)/2\geqslant 4$. By Theorem \ref{thm: k adm cycles min degree k}, $G$ contains $k$ admissible cycles whose lengths form an arithmetic progression of common difference two. If $G$ is non-bipartite, then a contradiction is caused by Theorems \ref{thm: 2-con-K3-k-cycles} and \ref{thm: 3-con-noK3-k-cycles}. So $G$ is bipartite. Recall that $k$ and $\ell$ are even. It clearly holds that any $k$ admissible cycles in a bipartite graph contain cycles of all even lengths modulo $k$, a contradiction.  
\end{proof}

\begin{proof}[\bf Proof of Theorem \ref{thm: l mod k min degree k-1}]
    Let $G$ be a counterexample to the theorem. By Theorem \ref{thm: 0 mod 4 min deg 2} and Lemma \ref{lemma: 2 mod 4 cycle}, $k\geqslant 6$. By Lemma \ref{lemma: high connectivity no l mod k cycle}, the connectivity of $G$ is at least $k/2\geqslant 3$. Let $H$ be a bipartite subgraph of $G$ such that the number of edges of $H$ is maximal. 
    It is not difficult to see that $H$ is connected and has minimum degree at least $\lceil(k-1)/2\rceil=k/2$. 

\setcounter{claim}{0}
\begin{claim}\label{claim: HNot2Connected}
    $H$ is not $2$-connected.    
\end{claim}

\begin{proof}
    Assume the opposite that $H$ is 2-connected. If $k=6$, then by Theorem~\ref{thm: ChenSaitoDeanetal}, $H$ has cycles of lengths 0, 1, 2 modulo 3 (which are of lengths 0, 4, 2 modulo 6 since $H$ is bipartite), or $H\cong K_{3,n-3}$. If $k\geqslant 8$, then by Theorem~\ref{thm: k adm cycles min degree k}, $H$ contains $k/2$ admissible cycles (which are of lengths $\ell$ modulo $k$ for all even $\ell$) or $H\cong K_{k/2,n-k/2}$. 
    
    Suppose now $H\cong K_{k/2,n-k/2}$, say with bipartition $(X,Y)$ such that $|X|=k/2$ and $|Y|=n-k/2$. This implies the existence of $(\ell\bmod k)$-cycles for any even $\ell\not\equiv 2\bmod k$. By $n\geqslant k+2$, we have $|Y|\geqslant k/2+2$. Since $\delta(G)\geqslant k-1$, there exists $y\in Y$ that has $k/2-1\geqslant 2$ neighbors in $G[Y]$ (recall that $k\geqslant 6$). Let $y_1$ and $y_2$ be two neighbors of $y$ in $Y$. By $H\cong K_{k/2,n-k/2}$ and $n\geqslant k+2$, there exists a $(y_1,y_2)$-path of length $k$ in $H-y$. Together with the edges $yy_1$ and $yy_2$, we get a cycle of length 2 modulo $k$ in $G$, a contradiction. 
\end{proof}

\begin{claim}\label{claim: B-z2connected}
    Let $B$ be an end block of $H$ and $c$ the cut-vertex of $H$ contained in $B$. Then $B-c$ is $2$-connected.
\end{claim}

\begin{proof}
    Suppose that $B-c$ has a cut-vertex $z$. Let $Q_1$, $Q_2$ be two components of $B-\{c,z\}$, and for $i=1,2$, let $H_i=H[V(Q_i)\cup\{c,z\}]$. Note that $H_i+cz$ is 2-connected and every vertex in $Q_i$ has degree at least $k/2$ in $H_i$. By Theorem \ref{thm: k adm paths}, $H_i$ contains $k/2-1$ admissible $(c,z)$-paths, $i=1,2$. It follows that $B$ contains $k-3\geqslant k/2$ admissible cycles, which are even cycles since $H$ is bipartite. Thus, $G$ contains a cycle of length $\ell$ modulo $k$, a contradiction.
\end{proof}

\begin{claim}\label{claim: OneCutVertex}
    $H$ has exactly one cut-vertex.
\end{claim}

\begin{proof}
    If $H$ has at least two cut vertices, then $H$ contains two disjoint end blocks. Let $B_1$, $B_2$ be two disjoint end blocks of $H$, and let $c_1,c_2$ two cut vertices of $H$ contained in $B_1,B_2$, respectively.

    \begin{subclaim}\label{claim: DisjointPathFromB1toB2}
        $G$ contains two disjoint paths $P^1,P^2$ from $B_1$ to $B_2$ such that $B_1\cup B_2\cup P^1\cup P^2$ is bipartite.
    \end{subclaim}

    \begin{proof}
        Recall that $G$ is 3-connected. Let $P_0^1,P_0^2,P_0^3$ be three disjoint paths from $B_1$ to $B_2$. Let $u_i^j$ be the end-vertex of $P_0^j$ in $B_i$, $j=1,2,3$, $i=1,2$. Then $B_i$ contains three internally disjoint paths from some vertex $u_i\in V(B_i)$ to $u_i^1, u_i^2, u_i^3$, respectively, say $P_i^j$, $j=1,2,3$, $i=1,2$. Now $\bigcup\{P_i^j: i\in\{0,1,2\}, j\in\{1,2,3\}\}$ is a $\varTheta$-graph, which contains an even cycle. We assume w.l.o.g. that $u_1P_1^1u_1^1P_0^1u_2^1P_2^1u_2P_2^2u_2^2P_0^2u_1^2P_1^2u_1$ is an even cycle. Since both $B_1$ and $B_2$ are bipartite, we see that $B_1\cup B_2\cup P_0^1\cup P_0^2$ is bipartite, as desired.
    \end{proof}

    Let $P^1,P^2$ be two disjoint paths of $G$ from $B_1$ to $B_2$ as in Claim \ref{claim: DisjointPathFromB1toB2}, and let $u_i^j$ be the end-vertex of $P^j$ in $B_i$, $j=1,2$, $i=1,2$. Set $G_1=B_1\cup B_2\cup P^1\cup P^2$. Note that every vertex of $B_i$ other than $c_i$ has degree at least $k/2$. By Theorem \ref{thm: k adm paths 2}, $B_i$ contains $k/2-1$ admissible $(u_i^1,u_i^2)$-paths. It follows that $G_1$ contains $k-3\geqslant k/2$ admissible cycles, which are $(\ell\bmod k)$-cycles for all even $\ell$ since $G_1$ is bipartite, a contradiction.
\end{proof}

By Claims \ref{claim: HNot2Connected} and \ref{claim: OneCutVertex}, we can assume that $c$ is the unique cut-vertex of $H$. So $c$ is contained in every block of~$H$.

\begin{claim}\label{claim: TwoBlocks}
    $H$ has exactly two blocks.
\end{claim}

\begin{proof}
    Suppose that $H$ has at least three blocks. Since $G$ is 3-connected, there exist blocks $B_0,B_1,B_2$ such that $E(G)\backslash E(H)$ contains two nonadjacent edges from $B_0-c$ to $B_1-c$ and $B_2-c$, respectively. Let $u_1v_1, u_2v_2$ be such two edges, and let $P$ be a path of $B_0-c$ from $u_1$ to $u_2$. 
    Let $G_2:=B_1\cup B_2\cup P\cup \{u_1v_1, u_2v_2\}$. For any path $P'$ in $B_1\cup B_2$ from $v_1$ to $v_2$, the cycle $u_2Pu_1v_1P'v_2u_2$ has all but exactly two edges in $H$ and thus is even. This implies that $G_2$ is bipartite. By Theorem \ref{thm: k adm paths}, $B_i$ contains $k/2-1$ admissible $(v_i,c)$-paths, for $i=1,2$. It follows that $G_2$ contains $k-3\geqslant k/2$ admissible cycles, which are $(\ell\bmod k)$-cycles for all even $\ell$, a contradiction.
\end{proof}

By Claim \ref{claim: TwoBlocks}, we can assume that $B_1$, $B_2$ are the two blocks of $H$. Recall that $G$ is 3-connected, there are two non-adjacent edges in $E(G)\backslash E(H)$ from $B_1-c$ to $B_2-c$, say $u_1u_2$, $v_1v_2$ with $u_1,v_1\in V(B_1)\backslash\{c\}$, $u_2,v_2\in V(B_2)\backslash\{c\}$. Let $H'=(B_1-c)\cup B_2\cup \{u_1u_2,v_1v_2\}$. For any two paths $P_1$ in $B_1-c$ from $u_1$ to $v_1$ and $P_2$ in $B_2$ from $u_2$ to $v_2$, the cycle $u_1P_1v_1v_2P_2u_2u_1$ is even. This implies that $H'$ is bipartite. By Claim \ref{claim: B-z2connected}, $B_1-c$ is 2-connected. It follows that $H'$ is 2-connected.

Suppose first that $k\geqslant 8$. Note that every vertex in $B_1-c$ has degree at least $k/2-1$. By Theorem \ref{thm: k adm paths}, $B_1-c$ contains $k/2-2$ admissible $(u_1,v_1)$-paths. Note that every vertex of $B_2$ other than $c$ has degree at least $k/2$ in $B_2$. By Theorem \ref{thm: k adm paths 2}, $B_2$ contains $k/2-1$ admissible $(u_1,v_1)$-paths. It follows that $H'$ contains $(k/2-2)+(k/2-1)-1=k-4\geqslant k/2$ admissible cycles, which are of length $\ell$ modulo $k$ for all even $\ell$, a contradiction. Thus, we conclude that $k=6$.

If $d_{B_1}(c)=2$, then $e(H')=e(H)$. Recall that $H'$ is 2-connected. By the analysis as in Claim \ref{claim: HNot2Connected}, $G$ contains cycles of lengths $\ell$ modulo $k$, a contradiction. So we conclude that $d_{B_1}(c)\geqslant 3$ and analogously $d_{B_2}(c)\geqslant 3$. That is, $\delta(B_1)\geqslant 3$ and $\delta(B_2)\geqslant 3$. 

By Theorem \ref{thm: ChenSaitoDeanetal}, $B_i$ ($i=1,2$) contains cycles of lengths 0, 1, 2 modulo 3 (which are of lengths 0, 4, 2 modulo~6), unless $B_i\cong K_{3,n_i-3}$, in which case $B_i$ contains no $(2\bmod 6)$-cycle, where $n_i=|V(B_i)|\geq 6$. Let $B'_1$ be a 4-cycle in $B_1-c$ containing $u_1,v_1$, and $B'_2$ be a $K_{3,3}$ in $B_2$ containing $u_2,v_2$. So $H''=B'_1\cup B'_2\cup\{u_1u_2,v_1v_2\}$ is a subgraph of $H'$. One can check that $H''$ contains a 8-cycle, a contradiction.
\end{proof}

\section{Cycles of lengths $\ell$ modulo $k$ for all $\ell$}\label{section: all ell}

In this section, we prove Theorem \ref{thm: l mod k min degree non-bi}. The following lemma is needed.

\begin{lemma}[Gao, Li, Ma and Xie \cite{GLMX}]\label{lemma: G-C is con}
    Let $G$ be a $3$-connected graph and $D$ a connected subgraph of $G$ such that $G-V(D)$ contains an even cycle. If $C$ is an even cycle of $G-V(D)$ such that the component of $G-V(C)$ containing $D$ is maximal, then $G-V(C)$ is connected. Moreover, such a $C$ has at most one chord or $G[V(C)]$ is isomorphic to $K_4$, and if $u_0v_0$ is a chord of $C$, then $d_C(u_0,v_0)$ is even.
\end{lemma}

Let $C$ be an even cycle. Recall that we say two vertices $u,v\in V(C)$ are diagonal if $d_C(u,v)=|C|/2$; and $u,v$ are quasi-diagonal if $d_C(u,v)=|C|/2-1$. Note that if $u,v$ are quasi-diagonal on $C$, then the two paths $\overrightarrow{C}[u,v]$ and $\overleftarrow{C}[u,v]$ have lengths differing by 2.

\begin{lemma}\label{lemma: TwoPathsTwoCycles}
    Let $G$ be a graph, $C_1$ an even cycle, and $C_2$ an odd cycle of $G$. If one of the following statements holds, then $G$ contains two cycles of consecutive odd lengths.
    \begin{itemize}
        \item[(1)] $C_1$ and $C_2$ are disjoint and there exist two disjoint paths $P^1$, $P^2$ from $C_1$ to $C_2$ such that their origins in $C_1$ are quasi‐diagonal.
        \item[(2)] $C_1$ and $C_2$ have exactly one common vertex $u$ and there exists a path $P$, from $C_1-u$ to $C_2-u$ avoiding $u$ such that its origin in $C_1$ is quasi‐diagonal to $u$.
    \end{itemize}
\end{lemma}

\begin{proof}
    For (1), let $u_i^j$ be the end-vertex of $P^j$ on $C_i$, where $j=1,2$ and $i=1,2$. Since $|C_2|$ is odd, the lengths of the two paths $R_1=P^1\cup P^2\cup\overrightarrow{C_2}[u_2^1,u_2^2]$ and $R_2=P^1\cup P^2\cup\overleftarrow{C_2}[u_2^1,u_2^2]$ differ in parity. Then one of these paths, say $R_1$, satisfies $|R_1|\equiv|C_1|/2\bmod 2$. Since $u_1^1,u_1^2$ are quasi‐diagonal in $C_1$, the two cycles $R_1\cup\overrightarrow{C_1}[u_1^1,u_1^2]$ and $R_1\cup\overleftarrow{C_1}[u_1^1,u_1^2]$ are of consecutive odd lengths. The proof for (2) is identical to (1).
\end{proof}

\begin{lemma}\label{lemma: TwoCycleEvenOdd}
    Let $G$ be a $3$-connected graph and $C_1, C_2$ two cycles of $G$. If one of the following statements holds, then $G$ contains two cycles of lengths $1$ and $3$ modulo $4$, respectively.
    \begin{itemize}
        \item[(1)] $C_1$ is even, $C_2$ is odd, and $C_1,C_2$ are disjoint.
        \item[(2)] $C_1$ is even, $C_2$ is odd, and $C_1,C_2$ intersect at exactly one vertex.
        \item[(3)] $C_1$ and $C_2$ are disjoint and both odd, and $G$ is not a Petersen graph.
    \end{itemize}
\end{lemma}

\begin{proof}
    (1) By Lemma \ref{lemma: G-C is con}, we can suppose that $G-V(C_1)$ is connected, $C_1$ has at most one chord, except that $G[V(C_1)]\cong K_4$, and if $u_0v_0$ is a chord of $C_1$, then $d_{C_1}(u_0,v_0)$ is even.

    If $|C_1|=4$, then there are three disjoint paths from $C_1$ to $C_2$, two of which have their origins quasi-diagonal in $C_1$. By Lemma \ref{lemma: TwoPathsTwoCycles}, there are two consecutive odd cycles, which are of lengths 1 and 3 modulo 4, as desired. So we conclude that $|C_1|\geqslant 6$, which follows that $C_1$ has at most one chord. 

    Set $H=G-V(C_1)$, and let $B$ be the block of $H$ containing $C_2$ (if $H$ is 2-connected, then $B=H$). Note that for each two vertices $v_1,v_2\in V(B)$, there are two disjoint paths from $\{v_1,v_2\}$ to $C_2$ in $B$. For every vertex $v\in V(B)$, we let $H_v$ be the component of $H-E(B)$ containing $v$. Possibly $V(H_v)=\{v\}$ (in the case $v$ is an inner-vertex of $B$). We call the graphs $H_v$, $v\in V(B)$, the $B$-branches.

    We claim that for each two vertices $u_1,u_2\in V(H)$, there are two disjoint paths from $\{u_1,u_2\}$ to $C_2$ in $H$ if and only if $u_1,u_2$ are contained in different $B$-branches. Firstly, if $u_1,u_2$ are contained in a common $B$-branch $H_v$, then every path from $\{u_1,u_2\}$ to $C_2$ passes through $v$. Secondly, suppose that $u_1,u_2$ are contained in $B$-branches $H_{v_1},H_{v_2}$, respectively, where $v_1\neq v_2$. Let $P_i$ be a $(u_i,v_i)$-path in $H_{v_i}$, $i=1,2$, and let $P'_1,P'_2$ be two disjoint paths from $\{v_1,v_2\}$ to $C_2$ in $B$. Now $P_1v_1P'_1$ and $P_2v_2P'_2$ are two disjoint paths from $\{u_1,u_2\}$ to $C_2$ in $H$.

    For two quasi-diagonal vertices $x_1,x_2$ in $C_1$, if $x_1$ has a neighbor $u_1$ in $H$, $x_2$ has a neighbor $u_2$ in $H$ such that $u_1,u_2$ are in different $B$-branches, then there are two disjoint paths from $C_1$ to $C_2$ with origins $x_1,x_2$. By Lemma \ref{lemma: TwoPathsTwoCycles}, there are two cycles of consecutive odd lengths. So we can assume that if both $x_1,x_2$ have neighbors in $H$, then their neighbors are contained in a common $B$-branch. 

    As in the proof of Lemma \ref{lemma: k adm cycles for k=4,5}, we define the graph $Qdi(C_1)$ on $V(C_1)$ such that two vertices are adjacent in $Qdi(C_1)$ if and only if they are quasi‐diagonal in $C_1$. Note that if $|C_1|\equiv 0\bmod 4$, then $Qdi(C_1)$ is a cycle, and if $|C_1|\equiv 2\bmod 4$, then $Qdi(C_1)$ is the union of two odd cycles. We consider two cases as follows.

    First we assume that $|C_1|\equiv 2\bmod 4$. If $C_1$ has a chord $u_1u_2$, then $d_{C_1}(u_1,u_2)$ is even. It follows that the two cycles $C'_1=\overrightarrow{C_1}[u_1,u_2]u_2u_1$ and $C'_2=\overleftarrow{C_1}[u_1,u_2]u_2u_1$ are odd. Since $|D_1|+|D_2|=|C_1|+2\equiv 0\bmod 4$, the two cycles $C'_1,C'_2$ have lengths 1 and 3 modulo 4, as desired. So we conclude that $C_1$ has no chord.
    By $G$ being 3-connected, every vertex in $C_1$ has a neighbor in $H$. Let $D_1,D_2$ be the two components of $Qdi(C_1)$. It follows that all vertices in $N_H(V(D_1))$ are in a common $B$-branch $H_{v_1}$, and all vertices in $N_H(V(D_2))$ are in a common $B$-branch $H_{v_2}$. Now $\{v_1,v_2\}$ is a vertex cut of $G$, a contradiction.

    Second we assume that $|C_1|\equiv 0\bmod 4$. Recall that $Qdi(C_1)$ is a cycle. If $C_1$ has no chord, then all vertices in $N_H(V(C_1))$ are in a common $B$-branch $H_v$, and $v$ is a cut-vertex of $G$, a contradiction. So we conclude that $C_1$ has a chord $u_1u_2$. Since $d_{C_1}(u_1,u_2)$ is even, we have $u_1u_2\notin E(Qdi(C_1))$.
    If $N_H(u_1)\cup N_H(u_2)\neq\emptyset$, then all the neighbors in $H$ of vertices in $C_1$ are in a common $B$-branch, implying that $G$ has a cut-vertex, a contradiction. So assume that $u_1$ and $u_2$ have no neighbor in $H$.
    Let $D_1,D_2$ be the two components of $Qdi(C_1)-\{u_1,u_2\}$. Then $D_1,D_2$ are two paths. If both of the paths are nontrivial, then $N_H(V(D_1))\subseteq V(H_{v_1})$, and $N_H(V(D_2))\subseteq V(H_{v_2})$ for some $B$-branches $H_{v_1}$, $H_{v_2}$. Now $\{v_1,v_2\}$ is a vertex cut of $G$, a contradiction.
    Now assume w.l.o.g. that $D_2$ is trivial, say $V(D_2)=\{u_0\}$. It follows that $N_H(V(D_1))\subseteq V(H_{v_1})$ for a $B$-branch $H_{v_1}$. Recall that $u_1,u_2$ have no neighbor in $H$. We have that $\{v_1,u_0\}$ is a vertex cut of $G$, a contradiction.    

    (2) Let $u$ be the common vertex of $C_1$ and $C_2$. We choose $C_1$ such that the component of $G-V(C_1)$ containing $C_2-u$ is maximal. By Lemma \ref{lemma: G-C is con}, $G-V(C_1)=:H$ is connected.

    Let $u_1,u_2$ be the two vertices that are quasi-diagonal with $u$ in $C_1$, and let $u_0$ be the vertex that is diagonal with $u$ in $C_1$. If $N_H(u_1)\cup N_H(u_2)\neq\emptyset$, then there is a path from $C_1-u$ to $C_2-u$ avoiding $u$ with origin in $\{u_1,u_2\}$. By Lemma \ref{lemma: TwoPathsTwoCycles}, $G$ has two consecutive odd cycles, as desired. So we conclude that $N_H(u_1)\cup N_H(u_2)=\emptyset$. If $|C_1|=4$, then $\{u,u_0\}$ is a vertex cut of $G$, contradicting the fact that $G$ is 3-connected. So we assume that $|C_1|\geqslant 6$. 
    
    By $G$ being 3-connected, we have that $u_1u_2$ is the unique chord of $C_1$. Let $u_3$ be a neighbor of $u$ on $C_1$. Then $N_H(u_3)\neq\emptyset$, and there is a path $P$ from $C_1-u$ to $C_2-u$ avoiding $u$ with origin $u_3$. Thus, $C_2\cup P\cup\{uu_3\}$ is a $\varTheta$-graph, which contains an even cycle disjoint with the odd cycle $u_0u_1u_2u_0$. Now we are done by (1).

    (3) We choose the two odd cycles $C_1,C_2$ such that the sum of its lengths is as small as possible. So neither $C_1$ nor $C_2$ has a chord. 
    
    If there is a component $H$ of $G-V(C_1)\cup V(C_2)$, then $|N_{C_1\cup C_2}(H)|\geqslant 3$ since $G$ is 3-connected. It follows that $H$ has at least two neighbors in $C_i$ for some $i=1,2$. Thus, there is a path $P$ with two end-vertices in $C_i$ and all internal vertices in $H$. We see that $C_i\cup P$ is a $\varTheta$-graph, which contains an even cycle disjoint with $C_{3-i}$, and we are done by (1). Now we assume that $G-V(C_1)\cup V(C_2)$ has no component, which is, $V(G)=V(C_1)\cup V(C_2)$.

    Let $u$ be an arbitrary vertex in $C_1$. If $u$ has two neighbors in $C_2$, then $G[V(C_2)\cup\{u\}]$ contains a $\varTheta$-graph, which contains an even cycle intersecting with $C_1$ at $u$. Then we are done by (2). So we conclude that every vertex in $C_1$ has exactly one neighbor in $C_2$, and similarly, every vertex in $C_2$ has exactly one neighbor in $C_1$. It follows that $|C_1|=|C_2|$, and the edges between $C_1$ and $C_2$ form a perfect matching of $G$.

    Let $u_1,u_2$ be two adjacent vertices in $C_1$, and $u'_1,u'_2$ their neighbors in $C_2$. If $u'_1u'_2\in E(C_2)$, then $C_1$ and $(C_1-\{u_1u_2\})\cup u_1u'_1u'_2u_2$ are two consecutive odd cycles, as desired. So we assume that $d_{C_2}(u'_1,u'_2)\geqslant 2$.

    If $|C_1|=|C_2|=3$, then every two vertices in $C_2$ have distance 1, a contradiction. If $|C_1|=|C_2|=5$, then for each two adjacent vertices in $C_1$, their neighbors in $C_2$ are of distance exactly 2, implying that $G$ is a Petersen graph, a contradiction. Therefore, $|C_1|=|C_2|\geqslant 7$.

    We choose two adjacent vertices $u_1,u_2\in V(C_1)$ such that $d_{C_2}(u'_1,u'_2)$ is as small as possible. We assume w.l.o.g. that $\overrightarrow{C_2}[u'_1,u'_2]$ is even and $\overleftarrow{C_2}[u'_1,u'_2]$ is odd. Now $C'_1=u_1u'_1\overrightarrow{C_2}[u'_1,u'_2]u'_2u_2u_1$ is an odd cycle. Since $\overleftarrow{C_2}[u'_1,u'_2]$ is odd and $u_1'$ and $u_2'$ are not adjacent, there are two adjacent vertices $v'_1,v'_2$ in $\overleftarrow{C_2}[u'_1,u'_2]$ different than $u_1'$ and $u_2'$. Let $v_1,v_2$ be the neighbors of $v'_1,v'_2$ in $C_1$. Clearly $\{v_1,v_2\}\cap\{u_1,u_2\}=\emptyset$. Let $P$ be the $(v_1,v_2)$-path in $C_1$ avoiding $u_1,u_2$. Then $C'_2=v'_1v_1Pv_2v'_2v'_1$ is a cycle disjoint with $C'_1$.

    If $C'_2$ is even, then we are done by (1). So we assume that $C'_2$ is odd. By the choice of $C_1,C_2$, we see that $V(G)=V(C'_1)\cup V(C'_2)$. It follows that $d_{C_2}(u'_1,u'_2)=3$ (i.e., $\overleftarrow{C_2}[u'_1,u'_2]=u'_1v'_1v'_2u'_2$), and $v_1,v_2$ are the two vertices adjacent to $u_1,u_2$ in $C_1$. Now $u_1v_1\in E(C_1)$ or $u_1v_2\in E(C_1)$ and their neighbors in $C_2$ have distance less than 3, contradicting the choice of $u_1,u_2$.
\end{proof}

\begin{lemma}\label{lemma: two consecutive odd cycles}
    Let $G$ be a $3$-connected non-bipartite graph. Then $G$ contains two cycles of lengths $1$ and $3$ modulo~$4$, respectively, unless $G$ is a $K_4$ or a Petersen graph.
\end{lemma}

\begin{proof}
    By $G$ being non-bipartite, let $C$ be the shortest odd cycle of $G$. Then $C$ has no chord. If $G-V(C)$ has a cycle, then we are done by Lemma \ref{lemma: TwoCycleEvenOdd}. So we assume that $G-V(C)$ is a forest. Let $T$ be a component of $G-V(C)$ with order as large as possible. Then $T$ is a tree.

    Suppose first that $|C|=3$. If $T$ is trivial, say $V(T)=\{u\}$, then $N(u)=V(C)$. By the choice of $T$, every component of $G-V(C)$ is trivial. If $G-V(C)$ has only one component $T$, then $G\cong K_4$, as desired. So let $u'$ be the vertex of a component of $G-V(C)$ other than $T$. We have $N(u')=V(C)$ as well. Clearly $G[V(C)\cup\{u,u'\}]$ contains a 3-cycle and a 5-cycle, as desired.

    Now assume that $T$ is nontrivial. Let $u_1,u_2$ be two leaves of $T$ and let $P$ be the $(u_1,u_2)$-path of $T$. Since $G$ is 3-connected, $u_i$ has two neighbors in $V(C)$, for $i=1,2$. It follows that $u_1,u_2$ have a common neighbor in $C$, say $v_1$. If $P$ is even, then $v_1u_1Pu_2v_1$ is an even cycle that intersects with $C$ at exactly one vertex, and we are done by Lemma \ref{lemma: TwoCycleEvenOdd}. So assume that $P$ is odd. Let $V(C)=\{v_1,v_2,v_3\}$ and assume $v_2$ is a second neighbor of $u_2$ in $C$. Then $C_1=v_1u_1Pu_2v_1$ and $C_2=v_1u_1Pu_2v_2v_3v_1$ are two consecutive odd cycles, as desired.

    Suppose now that $|C|\geqslant 5$. We claim that if a vertex in $T$ has two neighbors in $C$, then the two neighbors have distance exactly 2 in $C$. To show this, we assume that $u\in V(T)$ has two neighbors $v_1,v_2\in V(C)$. If $v_1v_2\in E(C)$, then $uv_1v_2u$ is a triangle, contradicting the choice of $C$. Now assume $d_C(v_1,v_2)\geqslant 3$. Note that either $\overrightarrow{C}[v_1,v_2]$ or $\overleftarrow{C}[v_1,v_2]$ is odd, say $\overrightarrow{C}[v_1,v_2]$ is odd. Since $|\overleftarrow{C}[v_1,v_2]|\geqslant 3$, $uv_1\overrightarrow{C}[v_1,v_2]v_2u$ is an odd cycle shorter than $C$, a contradiction. Thus, as we claimed, for every vertex in $T$, its two neighbors in $C$ have distance exactly 2 in $C$. It follows that every vertex in $T$ has at most two neighbors in $C$.
    If $T$ is trivial, then the vertex in $T$ has three neighbors in $C$, a contradiction. So we assume that $T$ is nontrivial. 
    
    Let $u_1$ be a leaf of $T$. Since $G$ is 3-connected, $u_1$ has two neighbors in $V(C)$. Let $v_1,v_2$ be the two neighbors of $u_1$ in $C$, and let $v_0$ be the common neighbor of $v_1,v_2$ in $C$ (recall that $d_C(v_1,v_2)=2$). Suppose that there is a vertex $v_3\in V(C)\backslash\{v_0,v_1,v_2\}$ that has a neighbor $u_3$ in $T$. We can assume w.l.o.g. that $v_1,v_2,v_3$ appear in this order along $C$. Let $P$ be the $(u_1,u_3)$-path in $T$. Note that either $C_1=\overrightarrow{C}[v_1,v_3]v_3u_3Pu_1v_1$ or $C_2=\overleftarrow{C}[v_1,v_3]v_3u_3Pu_1v_1$ is odd. If $C_1$ is odd, then $C_1$ and $C'_1=\overrightarrow{C}[v_2,v_3]v_3u_3Pu_1v_2$ are two consecutive odd cycles, as desired. While, if $C_2$ is odd, then $C_2$ and $C'_2=\overleftarrow{C}[v_2,v_3]v_3u_3Pu_1v_2$ are two consecutive odd cycles, as desired. So we conclude that there is no vertex in $V(C)\backslash\{v_0,v_1,v_2\}$ that has a neighbor in $T$. It follows that $N_C(T)=\{v_0,v_1,v_2\}$.

    Let $u_2$ be a leaf of $T$ other than $u_1$ and let $P$ be the $(u_1,u_2)$-path in $T$. Recall that $u_2$ has two neighbors in $C$ which have distance exactly 2, and that $N_C(T)=\{v_0,v_1,v_2\}$. This implies that $N_C(u_2)=\{v_1,v_2\}$. If $P$ is even, then $v_1u_1Pu_2v_1$ is an even cycle that intersects with $C$ at exactly one vertex, and we are done by Lemma~\ref{lemma: TwoCycleEvenOdd}. So assume that $P$ is odd. Then $C_1=v_1u_1Pu_2v_1$ and $C_2=v_1u_1Pu_2v_2v_0v_1$ are two consecutive odd cycles, as desired.     
\end{proof}

    We would like to remark that while writing this paper, we noticed that Lin, Wang and Zhou \cite{LWZ24} have showed that every 2-connected non-bipartite graph with minimum degree at least $k\geqslant 2$ contains $\lceil(k-1)/2\rceil$ cycles of consecutive odd lengths. Although Lemma \ref{lemma: two consecutive odd cycles} can be deduced from this result, our proof is different. For the sake of completeness, we kept the proof of Lemma \ref{lemma: two consecutive odd cycles} in this paper.

\begin{proof}[\bf Proof of Theorem \ref{thm: l mod k min degree non-bi}]
    Note that when $k$ is odd, the existence of cycles of all lengths modulo $k$ is equivalent to the existence of cycles of all even lengths modulo $k$. If $k=3$, then the assertion holds by Theorem \ref{thm: ChenSaitoDeanetal}. So assume that $k\geqslant 4$. By Theorem \ref{thm: l mod k min degree k}, it suffices to consider the case that $k$ is even and $\ell$ is odd. 
    
    If $G$ is 3-connected, then for $k=4$ the assertion holds by Lemma \ref{lemma: two consecutive odd cycles}. For $k \geqslant 6$, if $G$ contains a triangle, then we are done by Theorem \ref{thm: 2-con-K3-k-cycles}, while if $G$ does not contain a triangle, then we are done by Theorem \ref{thm: 3-con-noK3-k-cycles}. Therefore, we may assume that $G$ has a vertex cut with two vertices, say $\{x,y\}$.
    
    Denote by $Q_1,Q_2$ two components of $G-\{x,y\}$. Let $G_i=G[V(Q_i)\cup \{x,y\}]$, $i=1,2$. By Theorem \ref{thm: k adm paths}, there exist $k-1$ admissible $(x,y)$-paths in $G_i$ for each $i\in \{1,2\}$. If $G_i$ contains $k-1$ admissible $(x,y)$-paths of consecutive lengths for some $i\in \{1,2\}$, then we can get $2(k-1)-1\geqslant k$ cycles of consecutive lengths, which contain $k/2$ cycles of consecutive odd lengths in $G$, as desired. So $G_i$ has $k-1$ admissible $(x,y)$-paths of lengths forming an arithmetic progression with common difference two for any $i\in \{1,2\}$. 
    
    If $G_i$ contains an odd cycle for some $i\in \{1,2\}$, then there exist two $(x,y)$-paths with different parities in $G_i$. By concatenating the $k-1$ admissible $(x,y)$-paths with common difference two in $G_{3-i}$, we get $k-1$ cycles of consecutive odd lengths, as desired. So $G_i$ is bipartite for each $i\in \{1,2\}$. Since $G$ is non-bipartite, by symmetry, we can assume that $x,y$ are in different parts in $G_1$ and in the same part in $G_2$. It follows that $G_1$ has $k-1$ admissible $(x,y)$-paths of consecutive odd lengths and $G_2$ has $k-1$ admissible $(x,y)$-paths of consecutive even lengths, implying the existence of $2k-3\geqslant k$ cycles of consecutive odd lengths, as desired. 
\end{proof}

\section{Maximum number of edges}\label{section: edge bound}

The following two theorems on cycles in graphs will be used in the proof of Theorem~\ref{thm: general bound}.

\begin{theorem}[Woodall \cite{W72}]\label{thm: (n+3)/2}
    Every $n$-vertex graph with more than $\lfloor n^2/4\rfloor$ edges contains cycles of lengths from $3$ to $(n+3)/2$.
\end{theorem}

\begin{theorem}[Bondy \cite{Bon71}]\label{thm: pancyclic}
    Every $n$-vertex graph with minimum degree more than $n/2$ is pancyclic, i.e., contains cycles of lengths from $3$ to $n$.
\end{theorem}

We start with a useful lemma. 

\begin{lemma}\label{lem: from edge number to degree}
Let $a\geq 1$ be an integer and $G$ a graph on $n \geq 2a+1$ vertices satisfying $e(G) > a(n-a)$. Then $G$ contains cycles of all lengths from $3$ to $a+2$ or a $2$-connected subgraph $H$ such that $\delta(H) \geq a+1$. 
\end{lemma}

\begin{proof}
    We proceed by induction on $n$. If $n=2a+1$, then $n$ is odd and
   \[
     e(G)>a(n-a)
     =\frac{n-1}{2}\cdot \frac{n+1}{2}
     =\left\lfloor\frac{n}{2}\right\rfloor \cdot \left\lceil\frac{n}{2}\right\rceil
     =\left\lfloor\frac{n^2}{4}\right\rfloor.
   \]
   By Theorem \ref{thm: (n+3)/2}, we get that $G$ contains cycles of all lengths from $3$ to $a+2$, as wanted. 
   
    Assume now that $n>2a+1$ and the statement holds for any graph with order in $[2a+1,n-1]$. Suppose that $G$ contains a vertex $v$ of degree at most $a$. Then
\[
e(G-v)
=e(G)-d(v)
>a(n-a)-a
=a(n-1-a).
\]
    Thus, by induction hypothesis, $G-v$ contains wanted cycles or a subgraph, so does $G$. Therefore, we may assume that $\delta(G) \geq a+1$. 

    It is enough to consider the case when $G$ is connected, because otherwise we can add some cut edges, which do not influence existence of cycles or a 2-connected subgraph. 
    If $G$ is $2$-connected, then $H=G$ satisfies the thesis. Hence, let $c$ be a cut vertex of $G$. This means that $G$ is a union of two subgraphs $G_1,G_2$ with $V(G_1)\cap V(G_2)=\{c\}$. For $i=1,2$, let $n_i=|V(G_i)|$, where $n+1=n_1+n_2$. Since $\delta(G)\geq a+1$, $n_i\geq a+2$. If $n_i=a+2$, then $G_i \cong K_{a+2}$ and contains cycles of all lengths from $3$ to $a+2$, as wanted. So $n_i \geq a+3$.     
    If $n_i<2a+1$, then $\delta(G_i-c)\geq a>(n_i-1)/2$. By Theorem \ref{thm: pancyclic}, $G_i-c$ contains cycles of lengths from $3$ to $n_i-1 \geq a+2$, and we are done. So we may assume that $n_i\geq 2a+1$. By induction hypothesis $G_i$ contains wanted cycles or a subgraph, unless $e(G_i) \leq a(n_i-a)$. But then it follows that
    \[
     e(G)
     =e(G_1)+e(G_2)
     \leqslant a(n_1-a)+a(n_2-a)=a(n+1-2a) \leq a(n-a),
    \]
    a contradiction to the assumption that $e(G)>a(n-a)$.
\end{proof}

\begin{proof}[\bf Proof of Theorem \ref{thm: general bound}]
    As mentioned in the introduction, the result was proven earlier for $k=3$, hence we may assume that $k \geqslant 4$. 
    Consider first the case $\ell \not= 2$. Let $G$ be a minimal counterexample, that is a graph on $n \geq 2k-1$ vertices not containing $(\ell \bmod k)$-cycle satisfying $e(G) > (k-1)(n-k+1)$. By Lemma~\ref{lem: from edge number to degree} for $a=k-1$, $G$ contains a 2-connected subgraph $H$ with $\delta(H) \geq k$. 
    If $k$ is odd, then by Theorem~\ref{thm: k adm cycles min degree k}, $H$ contains $k$ admissible cycles, $H \cong K_{k+1}$, or $H \cong K_{k,|H|-k}$. In each case $H$, and so $G$, contains a cycle of length $\ell \bmod k$, because $\ell \not=2$ and $k$ is odd. 
    If $k$ is even, then by Theorem~\ref{thm: l mod k min degree k}, $H$ contains $(\ell \bmod k)$-cycle, $H \cong K_{k+1}$, or $H \cong K_{k,|H|-k}$. In each case, $G$ contains $(\ell \bmod k)$-cycle.  

    Now consider the case $\ell=2$. Similarly as before, let $G$ be a minimal counterexample, that is a graph on $n \geq 2k+1$ vertices not containing $(2 \bmod k)$-cycle satisfying $e(G) > k(n-k)$. By Lemma~\ref{lem: from edge number to degree} for $a=k$, $G$ contains a 2-connected subgraph $H$ with $\delta(H) \geq k+1$.
    By Theorem~\ref{thm: base} (2), $H$ contains a $(2 \bmod k)$-cycle, a contradiction.
\end{proof}

\section*{Acknowledgement}

The research of Yandong Bai and Binlong Li was supported by the National Key Research and Development Program of China (No. 2026YFE0151700), the National Natural Science Foundation of China (Nos. 12542043, 12242111, 12131013), and Guangdong Basic and Applied Basic Research Foundation (No. 2023A1515030208).  
The research of Andrzej Grzesik and Magdalena Prorok was supported by the National Science Centre (Grant No. 2021/42/E/ST1/00193). 
Part of this work was done during the 15th Emléktábla Workshop in Hungary and the hospitality was appreciated.

\end{spacing}
\end{document}